\documentclass[11pt]{amsart}

\usepackage[margin=1.15in]{geometry}
\usepackage{amsmath,amssymb,amsthm,mathtools}
\usepackage{mathrsfs}
\usepackage{enumitem}
\usepackage{hyperref}
\usepackage[nameinlink,capitalize]{cleveref}
\usepackage{tikz}
\usetikzlibrary{arrows.meta,positioning,calc,fit,decorations.pathmorphing,shapes.geometric}

\hypersetup{
  colorlinks=true,
  linkcolor=blue,
  citecolor=blue,
  urlcolor=blue,
  pdftitle={Ropelength-Filtered Swept-Area Geometry},
  pdfauthor={Makoto Ozawa}
}

\theoremstyle{plain}
\newtheorem{theorem}{Theorem}[section]
\newtheorem{proposition}[theorem]{Proposition}
\newtheorem{lemma}[theorem]{Lemma}
\newtheorem{corollary}[theorem]{Corollary}

\theoremstyle{definition}
\newtheorem{definition}[theorem]{Definition}
\newtheorem{question}[theorem]{Question}

\newtheorem{example}[theorem]{Example}

\theoremstyle{remark}
\newtheorem{remark}[theorem]{Remark}

\crefname{theorem}{Theorem}{Theorems}
\crefname{proposition}{Proposition}{Propositions}
\crefname{lemma}{Lemma}{Lemmas}
\crefname{corollary}{Corollary}{Corollaries}
\crefname{definition}{Definition}{Definitions}
\crefname{remark}{Remark}{Remarks}
\crefname{example}{Example}{Examples}
\crefname{question}{Question}{Questions}
\crefname{conjecture}{Conjecture}{Conjectures}

\newcommand{\R}{\mathbb{R}}

\newcommand{\Len}{\operatorname{Len}}
\newcommand{\Thi}{\operatorname{Thi}}
\newcommand{\Rop}{\operatorname{Rop}}
\newcommand{\Isom}{\operatorname{Isom}}
\newcommand{\Diff}{\operatorname{Diff}}

\newcommand{\diam}{\operatorname{diam}}
\newcommand{\TC}{\operatorname{TC}}
\newcommand{\dist}{\operatorname{dist}}
\newcommand{\CRad}{\operatorname{CRad}}
\newcommand{\Pol}{\operatorname{Pol}}

\newcommand{\eps}{\varepsilon}
\newcommand{\Flat}{\mathcal{F}}
\newcommand{\Fill}{\operatorname{Fill}}
\newcommand{\AreaVec}{\mathbf{A}}

\title[Ropelength-Filtered Swept-Area Geometry]{Ropelength-Filtered Swept-Area Geometry}

\author{Makoto Ozawa}
\address{Department of Natural Sciences, Faculty of Arts and Sciences, Komazawa University, Tokyo, Japan}
\email{w3c@komazawa-u.ac.jp}

\keywords{
knot type; ropelength; thickness; swept area; metric; flat norm;
admissible isotopy; knot space; polygonal knot; Reidemeister graph;
density; compression radius; swept-area length group
}

\subjclass[2020]{Primary 57K10; Secondary 53A04, 49Q10, 49Q15, 49Q20}
\begin{document}
\raggedbottom

\begin{abstract}
This paper studies the swept-area cost of isotopies between thick knot representatives when the isotopy is required to stay inside a ropelength window: every intermediate curve has thickness at least one and length at most $\Lambda$.  Without these constraints the swept area is the classical homotopy-area length on spaces of curves, and the induced distance on unparametrized curves is bounded below by the flat norm; see Yezzi--Mennucci and Michor--Mumford.  We record the corresponding non-degeneracy on the ropelength-filtered moduli space, taken modulo orientation-preserving reparametrizations and Euclidean isometries, as a consequence of this classical lower bound.

The ropelength window changes the theory in two ways.  Distances are infinite between classes that are not yet connected at the level $\Lambda$, and all costs depend monotonically on the budget.  We organize this dependence through budget--cost profiles and swept-area merge costs of admissible components.  We prove their monotonicity and a transport estimate under whole-path simulations, and relate them to the merge scales of ropelength-filtered knot spaces.  On the diagrammatic side, we construct a network with exact spatial endpoints whose path cost equals the infimal cost over diagrammatically generic isotopies, and show that the graph obtained by collapsing projection fibres gives only a lower bound, which can lose positive cost inside a fibre.  We also give projected-area calibrations, exact formulas for concentric round and homothetic elliptical unknots, in which the window is not active, a labelled polygonal estimate, and a based loop-length structure on admissible fundamental groups.  Existence of minimizing isotopies under the thickness and length constraints is left open.
\end{abstract}

\maketitle


\section{Introduction}

Spaces of embedded curves carry several natural geometric constraints.  In the ropelength setting one fixes a lower bound on thickness and imposes an upper bound on length.  These two requirements produce finite-scale deformation spaces: a curve may move only through embeddings that remain thick and whose length stays below a prescribed level.  Such constrained spaces are central in geometric knot theory, but their path geometry is usually studied only qualitatively, through the existence or non-existence of admissible isotopies.

The purpose of this paper is to put an intrinsic area-type geometry on these constrained deformation spaces.  If
\[
  \Gamma:S^{1}\times[0,1]\longrightarrow \R^{3}
\]
is an isotopy between embedded curves, its swept area is
\[
  A(\Gamma)=
  \int_{S^{1}\times[0,1]}
  \left|\partial_{s}\Gamma\times\partial_{t}\Gamma\right|\,ds\,dt,
\]
whenever the right-hand side is defined.  This is the parametrized area traced by the moving curve.  We minimize this quantity over isotopies satisfying thickness and length constraints and optimize the endpoint alignment over reparametrizations and rigid motions.  The result is an extended metric on the moduli space of unparametrized representatives.

For a knot type $K$, let
\[
  Y_{\Lambda}(K)=
  \{\gamma\in K\mid \Thi(\gamma)\ge 1,\ \Len(\gamma)\le \Lambda\}
  /\bigl(\Diff^{1,1}_+(S^1)\times\Isom^{+}(\R^{3})\bigr)
\]
denote the ropelength-filtered moduli space.  Its points are oriented,
unparametrized representatives modulo orientation-preserving Euclidean
isometries.  Here $\Len(\gamma)$ is length and $\Thi(\gamma)$ is thickness.
The ropelength of a representative is
\[
  \Rop(\gamma)=\frac{\Len(\gamma)}{\Thi(\gamma)},
\]
and the normalization $\Thi\ge 1$ turns the scale-invariant ropelength condition into a length bound.  We use the standard framework for thickness and ropelength; see Litherland et al., Gonzalez--Maddocks, and Cantarella--Kusner--Sullivan \cite{LSDR,GonzalezMaddocks,CKS}.  Quadrisecant lower bounds for ropelength are due to Denne--Diao--Sullivan \cite{DDS}.

The topology of $Y_{\Lambda}(K)$ records which representatives can be joined by thick isotopies below the length level $\Lambda$.  As $\Lambda$ increases, components may merge, leading to filtered-topological invariants such as admissible components, ideal strata, and merge scales.  The present paper refines this connectivity theory by assigning an area cost to admissible paths.  Thus the question is not only whether two representatives can be connected, but how much swept area is forced by any such connection.

There is also a group-theoretic form of the construction.  If $C$ is an admissible component and $\gamma_*$ is a base representative, admissible loops based at $\gamma_*$ form an admissible fundamental group.  The trace surface assigns to each group element an infimal swept-area length.  Thus the theory has two parallel aspects: a point-to-point cost $d_\Lambda^K$ on an admissible component and a swept-area length function on the admissible fundamental group.

\subsection*{Main results}
We first place the construction in the classical framework.  The swept area is the length of a homotopy for the Finsler metric given by the $L^1$ norm of the normal velocity, and the induced distance on unparametrized curves is bounded below by the flat norm \cite{YezziMennucci,MichorMumford,FedererFleming}.  Adapting this lower bound to the constrained setting and to the quotient by Euclidean isometries shows that $d_\Lambda^K$ is a genuine extended metric on $Y_\Lambda(K)$ and a finite metric on each admissible component (\cref{thm:swept-area-metric}).  We regard this as a standard consequence of the flat-norm lower bound rather than as a new phenomenon; see \cref{rem:classical-nondegeneracy}.

The features that depend on the ropelength window are the following.  First, $d_\Lambda^K$ is infinite between classes that are not admissibly connected at the level $\Lambda$, and it is nonincreasing in $\Lambda$.  The swept-area merge cost $a_\Lambda(C_0,C_1)$ of two admissible components therefore refines the merge scale $m(C_0,C_1)$ of the ropelength filtration by an area coordinate, and the immediate and relaxed post-merge costs $a_{m+}$ and $a_+$ need not agree (\cref{sec:ideal-strata}).  Second, budget--cost profiles record all attainable pairs of ropelength budget and swept-area cost; they transport under whole-path simulations with explicit budget and cost errors (\cref{sec:budget-cost}).  Third, a network with exact spatial endpoints computes the infimal cost over diagrammatically generic isotopies, whereas the graph obtained by collapsing projection fibres gives only a lower bound, which can lose positive cost inside a fibre (\cref{sec:weighted-reidemeister,sec:budget-cost}).

We also record computable calibrations from the projected-area vector, exact formulas for concentric round unknots,
\[
  d_{\Lambda}^{U}(C_1,C_R)=\pi(R^2-1),
  \qquad \Lambda\ge 2\pi R,
\]
and for homothetic ellipses, a labelled polygonal estimate on uniformly non-collinear strata, and a based loop-length structure on admissible fundamental groups.  In the exact formulas the ropelength window is not active: the radial homothety is admissible and the calibration lower bound holds for unconstrained isotopies.  These examples test the definitions; they do not exhibit a cost forced by the thickness constraint.  Producing such an example, and proving existence of minimizing isotopies under the constraints, are the main problems left open (\cref{sec:questions}).

The paper therefore has a three-layer structure.  Static quantities assign numerical invariants to representatives.  Filtered topology records how thick representatives become connected as the ropelength bound is relaxed.  Dynamic geometry assigns an area cost to admissible paths.

\subsection*{Relation to earlier work}
The swept-area functional is classical.  In the theory of metrics on spaces of curves, the area $\int\!\!\int|\partial_s\Gamma\times\partial_t\Gamma|$ is the length of a homotopy for the Finsler metric $\|\pi_N h\|_{L^1}$ on curves in $\R^n$, $n\ge2$; Yezzi and Mennucci discuss it as a basic example, note its relation with the Plateau problem, and analyse the degeneracy caused by reparametrization \cite{YezziMennucci}.  Michor and Mumford explain why distances on curves should be taken on the quotient by reparametrizations, and point out that such area distances are bounded below by the flat norm \cite{MichorMumford}.  The flat norm itself goes back to Whitney and to Federer--Fleming \cite{Whitney,FedererFleming,Federer}.  Minimum-area homotopies between curves on surfaces, with algorithms, are studied by Chambers--Wang and in subsequent computational work \cite{ChambersWang,ChangEtAl}.  In particular, the non-degeneracy of \cref{thm:swept-area-metric} is an adaptation of this classical lower bound, and we do not claim it as new.

The fact that thickness and length constraints can separate curves more finely than knot type is also classical in geometric knot theory.  Nabutovsky's work on knots with thick ropes \cite{Nabutovsky}, the numerical Gordian unknots of Pieranski, Przyby\l{} and Stasiak \cite{PieranskiGordian}, and the rigorous Gordian split links of Coward and Hass \cite{CowardHass} are the basic references.  What the present paper adds is an area coordinate on this constrained connectivity: costs that depend on the ropelength budget, merge costs of admissible components, and their diagrammatic counterparts.

Related $p$-density, compression-radius, deformation-persistence, and finite-recognition viewpoints are developed in the author's preprints \cite{OzawaUnifiedPDensities,OzawaIdealStratum,OzawaFiniteKnotTheory}.  No proof here depends on an unpublished theorem from those preprints.

\begin{figure}[t]
\centering
\begin{tikzpicture}[
  layer/.style={draw, rounded corners=2pt, minimum width=0.74\linewidth, minimum height=0.85cm, align=center, fill=gray!8},
  small/.style={font=\small},
  arr/.style={-{Latex[length=2mm]}, thick}
]
\node[layer] (static) {\textbf{Static geometry}\quad functions on representatives\\
  $\rho_D(\gamma)=\Len(\gamma)/D(\gamma)$,\quad $\CRad_D(\gamma)=D(\gamma)/\Thi(\gamma)$};
\node[layer, below=0.55cm of static] (filtered) {\textbf{Filtered topology}\quad admissible sublevel spaces\\
  $I(K)\subset Y_{\Lambda_1}(K)\subset Y_{\Lambda_2}(K)\subset\cdots$};
\node[layer, below=0.55cm of filtered] (dynamic) {\textbf{Dynamic geometry}\quad costs of paths\\
  $d_\Lambda^K(x_0,x_1)=\inf_\Gamma A(\Gamma)$};
\draw[arr] (static) -- (filtered);
\draw[arr] (filtered) -- (dynamic);
\end{tikzpicture}
\caption{Static invariants, filtered topology, and swept-area cost: the paper studies how static invariants of representatives, filtered connectivity, and dynamic swept-area cost interact.}
\label{fig:three-layer}
\end{figure}

This structure is summarized in Figure~\ref{fig:three-layer}.

The same idea also has a diagrammatic counterpart.
By taking generic projections and recording Reidemeister moves that lift to admissible thick deformations, one obtains diagram graphs filtered by ropelength.
The swept-area functional assigns weights to the realized transitions, thereby turning these graphs into weighted objects.
This should be compared with the ropelength-filtered lifted Reidemeister graphs and characteristic Reidemeister patterns of \cite{OzawaFiniteKnotTheory}, whose vertices retain components of regular projection fibres; the auxiliary graph used here collapses those fibres and is defined explicitly in \cref{sec:weighted-reidemeister}.  In this way, the present framework connects ropelength geometry, deformation persistence, and diagrammatic transition costs.

\subsection*{Organization of the paper}
In \cref{sec:ropelength-spaces} we recall ropelength-filtered knot spaces and admissible deformations.
In \cref{sec:swept-area,sec:metric} we define swept area, construct the associated extended metric, and prove non-degeneracy by the flat norm of integral currents.
In \cref{sec:admissible-fundamental-groups} we define admissible fundamental groups and swept-area length functions.
In \cref{sec:three-layer,sec:ideal-strata} we relate the construction to static invariants, filtered topology, and merge costs, carefully distinguishing proved estimates from conditional principles and open questions.
In \cref{sec:polygonal} we prove a quantitative finite-dimensional theorem in labelled polygonal models.
In \cref{sec:weighted-reidemeister} we introduce swept-area weighted diagram graphs and compare their diagrammatic distances with generic geometric swept-area cost.
In \cref{sec:examples} we prove the calibration estimates and exact computations for planar families, and record the standard description of the ideal unknot.
In \cref{sec:budget-cost} we study budget--cost profiles, their transport under whole-path simulations, and a network with exact spatial endpoints whose path cost equals the infimal generic geometric cost.
The final section records a short list of analytic and geometric questions needed for a fuller theory.

\section{Ropelength-filtered knot spaces}
\label{sec:ropelength-spaces}

\subsection{Regularity conventions}
\label{subsec:regularity-conventions}
Throughout the paper we use the following convention.  In the smooth category, representatives are closed embedded $C^{1,1}$ curves.  Thus length, curvature almost everywhere, and thickness are defined in the standard ropelength sense.  An admissible isotopy is required to be Lipschitz as a map $S^1\times[0,1]\to\R^3$, with each time slice a $C^{1,1}$ embedded representative of the fixed knot type satisfying the stated length and thickness bounds.  This is the regularity needed for the area formula used in \cref{sec:swept-area}.

In the polygonal category, representatives are labelled embedded polygons with a fixed number of edges, and admissible polygonal isotopies are absolutely continuous paths of labelled vertices.  The polygonal category is treated separately in \cref{sec:polygonal}.  It is used in this paper as an independent finite-dimensional analogue of the swept-area construction, not as an approximation theorem for the full $C^{1,1}$ theory.  Unless explicitly stated otherwise, statements involving $C^{1,1}$ representatives are not meant to include changes of polygonal subdivision, and polygonal statements are fixed-$N$ statements.

Let $K$ be an oriented knot type in $\R^{3}$.
Throughout this paper, a parametrized representative of $K$ means a closed embedded curve
\[
  \gamma:S^{1}\longrightarrow \R^{3}
\]
of knot type $K$, taken in the $C^{1,1}$ regularity class fixed in \cref{subsec:regularity-conventions}, unless the polygonal setting is explicitly invoked.  Orientation-preserving reparametrizations represent the same oriented curve.  For an unoriented knot type one may additionally divide by orientation reversal; the metric proof below is unchanged after taking the extra two-element orbit.

The thickness of $\gamma$, denoted by $\Thi(\gamma)$, is the supremal radius of an embedded normal tube around $\gamma$; equivalently, in the usual $C^{1,1}$ setting it may be described in terms of curvature radius and doubly-critical self-distance \cite{LSDR,GonzalezMaddocks}.
Its ropelength is
\[
  \Rop(\gamma)=\frac{\Len(\gamma)}{\Thi(\gamma)}.
\]
The ropelength of the knot type is
\[
  \Rop(K)=\inf\{\Rop(\gamma)\mid \gamma\text{ represents }K\}.
\]
Since ropelength is scale invariant, it is natural to normalize thickness.

\begin{definition}[Ropelength-filtered knot space]
For $\Lambda>0$, define
\[
  \widetilde{Y}_{\Lambda}(K)=
  \{\gamma\in K\mid \Thi(\gamma)\ge 1,\ \Len(\gamma)\le \Lambda\}.
\]
Let $\Diff^{1,1}_+(S^1)$ denote the group of orientation-preserving
$C^{1,1}$ diffeomorphisms with $C^{1,1}$ inverses.  This regularity keeps
the $C^{1,1}$ curve class invariant.  The group
$\Diff^{1,1}_+(S^1)\times\Isom^+(\R^3)$ acts by
\[
  (\phi,g)\cdot\gamma=g\circ\gamma\circ\phi.
\]
The corresponding moduli space of oriented unparametrized curves is
\[
  Y_{\Lambda}(K)=
  \widetilde{Y}_{\Lambda}(K)/
  \bigl(\Diff^{1,1}_+(S^1)\times\Isom^{+}(\R^{3})\bigr).
\]
\end{definition}

The spaces $Y_{\Lambda}(K)$ form an increasing filtration:
\[
  Y_{\Lambda}(K)\subset Y_{\Lambda'}(K)
  \qquad
  \text{whenever }\Lambda\le \Lambda'.
\]
They are empty for $\Lambda<\Rop(K)$.  Ropelength minimizers exist in every tame knot type in the $C^{1,1}$ category \cite{CKS}; after scaling a minimizer to thickness one, the level $\Lambda=\Rop(K)$ is therefore nonempty and is the first nonempty level.

\begin{figure}[t]
\centering
\begin{tikzpicture}[scale=0.95, every node/.style={font=\small}]
  \draw[thick, rounded corners=18pt, fill=gray!6] (-3.6,-1.75) rectangle (3.6,1.75);
  \draw[thick, rounded corners=14pt, fill=white] (-2.7,-1.15) rectangle (2.7,1.15);
  \draw[thick, fill=gray!15] (-1.35,-0.05) ellipse (0.72 and 0.42);
  \draw[thick, fill=gray!15] (1.35,-0.05) ellipse (0.72 and 0.42);
  \node at (0,1.45) {$Y_{\Lambda_2}(K)$};
  \node at (0,0.82) {$Y_{\Lambda_1}(K)$};
  \node[font=\footnotesize,left] at (-2.07,-0.05) {$C_0$};
  \node[font=\footnotesize,right] at (2.07,-0.05) {$C_1$};
  \fill (-1.35,0.0) circle (1.5pt) node[above=1pt,font=\footnotesize] {$\gamma_0$};
  \fill (1.35,0.0) circle (1.5pt) node[above=1pt,font=\footnotesize] {$\gamma_1$};
  \draw[-{Latex[length=2mm]}, thick, decorate, decoration={snake, amplitude=.35mm, segment length=3mm, post length=2mm}]
    (-1.3,-0.08) .. controls (-0.7,-0.95) and (0.7,-0.95) .. (1.28,-0.1);
  \node[font=\footnotesize,align=center] at (0,-2.05) {a $\Lambda_1$-admissible path joining two components $C_0,C_1$ of $I(K)$};
  \node[anchor=west] at (3.85,1.45) {$\Rop(K)<\Lambda_1<\Lambda_2$};
\end{tikzpicture}
\caption{Ropelength-filtered knot spaces.  The ideal stratum \(I(K)\) may have several components \(C_0,C_1\); at a higher level \(\Lambda_1\) an admissible path can join them.  The swept-area metric measures the cost of such paths.}
\label{fig:filtration}
\end{figure}

\begin{definition}[Ideal stratum]
The ideal stratum of $K$ is
\[
  I(K)=Y_{\Rop(K)}(K).
\]
Equivalently, it is the space of thickness-one representatives of $K$ with length equal to the ropelength of $K$, modulo orientation-preserving reparametrizations and Euclidean isometries.
\end{definition}

In practice, the ideal stratum may have more than one point or more than one component.
The geometry of these components is one motivation for the swept-area construction.

\begin{definition}[Admissible isotopy]
Let $\gamma_{0},\gamma_{1}\in \widetilde{Y}_{\Lambda}(K)$.
An isotopy
\[
  \Gamma:S^{1}\times[0,1]\longrightarrow \R^{3},
  \qquad
  \Gamma(\cdot,0)=\gamma_{0},\quad
  \Gamma(\cdot,1)=\gamma_{1},
\]
is called $\Lambda$-admissible if $\Gamma$ has the regularity specified in \cref{subsec:regularity-conventions} and, for every $t\in[0,1]$, the curve
\[
  \Gamma_{t}=\Gamma(\cdot,t)
\]
satisfies
\[
  \Thi(\Gamma_{t})\ge 1,
  \qquad
  \Len(\Gamma_{t})\le \Lambda.
\]
\end{definition}

Thus a $\Lambda$-admissible isotopy is an isotopy through representatives which remain inside the ropelength sublevel space.

\begin{definition}[Admissible component]
Two classes $x_0,x_1\in Y_{\Lambda}(K)$ are $\Lambda$-admissibly equivalent if there are representatives $\gamma_i\in x_i$ joined by a $\Lambda$-admissible isotopy.  The resulting equivalence classes are called the admissible components of $Y_{\Lambda}(K)$.
\end{definition}

This is indeed an equivalence relation.  For transitivity, two paths whose
middle endpoints are different representatives of the same moduli class are
aligned by applying the same reparametrization and Euclidean isometry to one
entire path, and are then concatenated.

\begin{lemma}[Simultaneous symmetry invariance]
\label{lem:simultaneous-invariance}
Let $g\in\Isom^+(\R^3)$, let $\phi\in\Diff^{1,1}_+(S^1)$, and let $\Gamma$ be an admissible isotopy.  Then
\[
  \widehat\Gamma(s,t)=g\bigl(\Gamma(\phi(s),t)\bigr)
\]
is admissible whenever $\Gamma$ is, and
\[
  A(\widehat\Gamma)=A(\Gamma).
\]
\end{lemma}

\begin{proof}
Euclidean isometries preserve length, thickness, and cross-product norms.
Changing the parameter preserves every time-slice image.  Its Jacobian
factor is cancelled by change of variables in $s$.  Thus admissibility and
swept area are unchanged.
\end{proof}

\section{Swept area of knot isotopies}
\label{sec:swept-area}

We now define the area swept out by a moving knot.
Let
\[
  \Gamma:S^{1}\times[0,1]\longrightarrow \R^{3}
\]
be a Lipschitz isotopy as a map from the parameter cylinder to $\R^3$, or a $C^1$ isotopy when smooth calculations are being made.
We write $\Gamma_{t}(s)=\Gamma(s,t)$.

\begin{definition}[Swept area]
The swept area of $\Gamma$ is
\[
  A(\Gamma)=
  \int_{S^{1}\times[0,1]}
  |\partial_{s}\Gamma(s,t)\times \partial_{t}\Gamma(s,t)|\,ds\,dt.
\]
\end{definition}

Geometrically, $A(\Gamma)$ is the parametrized area of the trace surface swept by the moving curve.
The integrand is the Jacobian of the map $\Gamma$ from the parameter cylinder to $\R^{3}$.

\begin{figure}[t]
\centering
\begin{tikzpicture}[scale=0.95, every node/.style={font=\small}]
  \draw[fill=gray!12, draw=none] plot[smooth, tension=.8] coordinates {(-3,-1.2) (-1.2,-0.9) (0.6,-1.35) (3,-1.2)} --
    plot[smooth, tension=.8] coordinates {(3,1.0) (1.1,0.72) (-0.7,1.28) (-2.2,1.1)} -- cycle;
  \draw[thick] plot[smooth, tension=.8] coordinates {(-3,-1.2) (-1.2,-0.9) (0.6,-1.35) (3,-1.2)} node[right] {$\gamma_0=\Gamma_0$};
  \draw[thick] plot[smooth, tension=.8] coordinates {(-2.2,1.1) (-0.7,1.28) (1.1,0.72) (3,1.0)} node[right] {$\gamma_1=\Gamma_1$};
  \draw[gray, thin] plot[smooth, tension=.8] coordinates {(-2.62,-0.05) (-1.0,0.1) (0.9,-0.3) (3,-0.1)} node[right,text=black] {$\Gamma_t$};
  \foreach \a/\b in {{(-2.5,-0.8)}/{(-2.0,0.72)},{(-0.3,-1.05)}/{(-0.05,0.83)},{(2.1,-0.95)}/{(2.5,0.68)}}
    \draw[-{Latex[length=2mm]}, thick] \a -- \b;
  \node[anchor=south] at (-2.85,1.4) {trace surface};
  \node at (0.4,-2.0) {$A(\Gamma)=\displaystyle\int_{S^1\times[0,1]} |\partial_s\Gamma\times\partial_t\Gamma|\,ds\,dt$};
\end{tikzpicture}
\caption{The trace surface of a knot isotopy \(\Gamma\) from \(\gamma_0\) to \(\gamma_1\).  Arrows indicate the time direction; the swept area \(A(\Gamma)\) is the parametrized area of this surface.}
\label{fig:swept-area}
\end{figure}

\begin{remark}[Regularity and well-definedness]
The formula is immediate for $C^1$ isotopies.  More generally, throughout the paper we allow isotopies which are Lipschitz as maps
$S^1\times[0,1]\to\R^3$ and whose time slices are embedded curves in the chosen knot class.  For such maps the weak derivatives exist almost everywhere and the area formula of Federer \cite[Section~3.2.3]{Federer} gives a well-defined parametrized area; see also Morgan \cite{Morgan} for an accessible geometric-measure-theoretic reference
\[
  A(\Gamma)=\int_{S^1\times[0,1]} J_2\Gamma\,ds\,dt,
  \qquad
  J_2\Gamma=|\partial_s\Gamma\times\partial_t\Gamma|.
\]
For $C^{1,1}$ curves this is compatible with the usual ropelength setting, and for polygonal isotopies it agrees with the area of the associated piecewise-linear trace surface counted with multiplicity.  No existence of a minimizing isotopy is assumed in the definition of $d_\Lambda^K$.
\end{remark}

\begin{remark}[Self-intersections and multiplicity]
Although each time slice of an admissible isotopy is an embedded curve, the trace map $\Gamma:S^1\times[0,1]\to\R^3$ need not be an embedding of the cylinder.  The trace surface may have self-intersections.  The quantity $A(\Gamma)$ is the parametrized area, counted with multiplicity.  This is the appropriate notion for the calibration estimates below, because for every smooth two-form $\omega$ of comass at most one we have the pointwise bound $|\Gamma^*\omega|\le J_2\Gamma$ almost everywhere.
\end{remark}

\begin{remark}[Reparametrization]
The quantity $A(\Gamma)$ is invariant even under a time-dependent orientation-preserving reparametrization of the source circle.  Let $\phi_t\in\Diff^{1,1}_+(S^1)$ be a jointly Lipschitz family, with the derivatives below understood almost everywhere, and set
\[
  \widehat\Gamma(s,t)=\Gamma(\phi_t(s),t).
\]
Then the chain rule gives
\[
  \partial_s\widehat\Gamma=(\partial_s\phi_t)\,\partial_s\Gamma,
  \qquad
  \partial_t\widehat\Gamma=(\partial_t\phi_t)\,\partial_s\Gamma+\partial_t\Gamma,
\]
where the derivatives of $\Gamma$ on the right are evaluated at $(\phi_t(s),t)$.  The tangential term in $\partial_t\widehat\Gamma$ disappears after taking the cross product, and the remaining factor $|\partial_s\phi_t|$ is cancelled by change of variables in $s$.  Hence $A(\widehat\Gamma)=A(\Gamma)$.  Likewise, a monotone $C^1$ time reparametrization fixing the endpoints does not change the area.  This explains why the functional naturally lives on unparametrized curves.  We do not use non-injective reparametrizations.
\end{remark}

\begin{lemma}[Basic properties of swept area]
Let $\Gamma$ be an isotopy from $\gamma_{0}$ to $\gamma_{1}$.
Then:
\begin{enumerate}[label=\textup{(\arabic*)}]
  \item $A(\Gamma)\ge 0$.
  \item If $\overline{\Gamma}(s,t)=\Gamma(s,1-t)$, then
  \[
    A(\overline{\Gamma})=A(\Gamma).
  \]
  \item If $\Gamma^{1}$ is an isotopy from $\gamma_{0}$ to $\gamma_{1}$ and $\Gamma^{2}$ is an isotopy from $\gamma_{1}$ to $\gamma_{2}$, then their concatenation satisfies
  \[
    A(\Gamma^{1}*\Gamma^{2})=A(\Gamma^{1})+A(\Gamma^{2}).
  \]
\end{enumerate}
\end{lemma}

\begin{proof}
The first assertion follows from the non-negativity of the integrand.
The second follows by the change of variables $t\mapsto 1-t$.
The third follows by splitting the parameter interval into two subintervals and reparametrizing each isotopy linearly in time.
\end{proof}

\section{The swept-area extended metric}
\label{sec:metric}

We now define the main object of the paper.
Fix a knot type $K$ and a ropelength level $\Lambda\ge \Rop(K)$.

\begin{definition}[Exact-endpoint and quotient swept-area costs]
Let $\gamma_{0},\gamma_{1}\in \widetilde{Y}_{\Lambda}(K)$.
Define
\[
  \widetilde d_{\Lambda}^{K}(\gamma_{0},\gamma_{1})=
  \inf_{\Gamma} A(\Gamma),
\]
where the infimum is taken over all $\Lambda$-admissible isotopies $\Gamma$ from $\gamma_{0}$ to $\gamma_{1}$.
If there is no such isotopy, the value is $+\infty$.

For classes $x_i=[\gamma_i]\in Y_\Lambda(K)$, define the quotient cost by
\[
  d_\Lambda^K(x_0,x_1)
  =\inf_{\substack{g\in\Isom^+(\R^3)\\
                    \phi\in\Diff^{1,1}_+(S^1)}}
   \widetilde d_\Lambda^K
   \bigl(\gamma_0,g\circ\gamma_1\circ\phi\bigr).
\]
\end{definition}

Equivalently, $d_\Lambda^K(x_0,x_1)$ is the infimum of $A(\Gamma)$ over admissible isotopies whose initial and terminal curves represent $x_0$ and $x_1$, respectively.  The explicit endpoint-alignment infimum is essential: simultaneous invariance alone would not make an exact-endpoint cost independent of two separately chosen orbit representatives.  By \cref{lem:simultaneous-invariance}, the displayed value is independent of the representatives $\gamma_i$.

\begin{definition}[Extended pseudometric]
A function
\[
  d:X\times X\longrightarrow [0,+\infty]
\]
is called an extended pseudometric if, for all $x,y,z\in X$,
\[
  d(x,x)=0,
  \qquad
  d(x,y)=d(y,x),
\]
and
\[
  d(x,z)\le d(x,y)+d(y,z).
\]
Unlike a metric, it is not required that $d(x,y)=0$ imply $x=y$.
An extended metric is an extended pseudometric which also satisfies this
separation axiom.
\end{definition}

\begin{proposition}[Formal extended-pseudometric property]
\label{prop:pseudometric-formal}
Let $K$ be a knot type and let $\Lambda\ge \Rop(K)$.
The quotient swept-area cost $d_{\Lambda}^{K}$ is an extended pseudometric on all of $Y_{\Lambda}(K)$.  Its restriction to each admissible component is finite.
\end{proposition}

\begin{proof}
The constant isotopy gives $d_\Lambda^K(x,x)=0$.  Time reversal, together with simultaneous invariance under the symmetry group, gives symmetry.  If either term on the right-hand side of the triangle inequality is $+\infty$, there is nothing to prove.

For the triangle inequality, choose admissible paths from a representative $\gamma_0$ of $x_0$ to an aligned representative $g_{01}\circ\gamma_1\circ\phi_{01}$ of $x_1$, and from $\gamma_1$ to $g_{12}\circ\gamma_2\circ\phi_{12}$, with areas within $\eps$ of the two quotient infima.  Apply $g_{01}$ and precompose by $\phi_{01}$ throughout the second path.  Its initial curve then agrees with the terminal curve of the first path, its terminal curve still represents $x_2$, and its area is unchanged.  Concatenation gives
\[
  d_\Lambda^K(x_0,x_2)
  \le d_\Lambda^K(x_0,x_1)+d_\Lambda^K(x_1,x_2)+2\eps.
\]
Letting $\eps\downarrow0$ proves the triangle inequality.  Finally, finite distance is equivalent to belonging to the same admissible component.
\end{proof}

We now prove separation.  For an oriented Lipschitz closed curve $\gamma$, let
$T_\gamma=\gamma_\#[\![S^1]\!]$ be the associated integral $1$-current.
For an integral $1$-current $T$ in $\R^3$, write
\[
  \Flat(T)=\inf\{\mathbf M(R)+\mathbf M(S)
  \mid T=R+\partial S,\ R\in\mathbf I_1(\R^3),\
  S\in\mathbf I_2(\R^3)\}
\]
for its Federer--Fleming flat norm, where $\mathbf M$ denotes mass and
$\mathbf I_j(\R^3)$ denotes the space of integral $j$-currents
\cite{Federer,Morgan}.  When $\partial T=0$, also write
\[
  \Fill(T)=\inf\{\mathbf M(S)\mid
  S\in\mathbf I_2(\R^3),\ \partial S=T\}
\]
for its filling area.  Define the Euclidean-orbit versions by
\[
  \Flat_{\mathrm{orb}}(x_0,x_1)
  =\inf_{g\in\Isom^+(\R^3)}
  \Flat\bigl(T_{\gamma_0}-g_\#T_{\gamma_1}\bigr),
\]
and
\[
  \Fill_{\mathrm{orb}}(x_0,x_1)
  =\inf_{g\in\Isom^+(\R^3)}
  \Fill\bigl(T_{\gamma_0}-g_\#T_{\gamma_1}\bigr),
  \qquad x_i=[\gamma_i].
\]
Orientation-preserving reparametrization does not change $T_\gamma$, and
both mass functionals are invariant under Euclidean isometries.  Thus the
two orbit quantities are well defined on $Y_\Lambda(K)$.

\begin{lemma}[Filling and flat-current lower bounds]
\label{lem:flat-current-lower-bound}
For $\gamma_0,\gamma_1\in\widetilde Y_\Lambda(K)$,
\[
  \widetilde d_\Lambda^K(\gamma_0,\gamma_1)
  \ge \Fill(T_{\gamma_1}-T_{\gamma_0})
  \ge \Flat(T_{\gamma_1}-T_{\gamma_0}).
\]
Consequently, for $x_0,x_1\in Y_\Lambda(K)$,
\[
  d_\Lambda^K(x_0,x_1)
  \ge\Fill_{\mathrm{orb}}(x_0,x_1)
  \ge
  \Flat_{\mathrm{orb}}(x_0,x_1).
\]
\end{lemma}

\begin{proof}
Let $\Gamma$ be an admissible isotopy with the indicated exact endpoints.
The push-forward
\[
  S_\Gamma=\Gamma_\#[\![S^1\times[0,1]]\!]
\]
is an integral $2$-current and
\[
  \partial S_\Gamma=T_{\gamma_1}-T_{\gamma_0}.
\]
The area formula gives
\[
  \mathbf M(S_\Gamma)\le
  \int_{S^1\times[0,1]}J_2\Gamma=A(\Gamma);
\]
The pushforward mass can be smaller because of cancellation.  In contrast,
$A(\Gamma)$ counts parametrized multiplicity.  Taking the infimum over all
fillings gives $\Fill(T_{\gamma_1}-T_{\gamma_0})\le A(\Gamma)$, while the
choice $R=0$ in the flat-norm definition gives $\Flat\le\Fill$.  Finally,
take the infimum over isotopies and then optimize the endpoint alignment.
\end{proof}

\begin{lemma}[Separation of Euclidean orbits]
\label{lem:flat-orbit-separation}
For two oriented embedded Lipschitz closed curves $\gamma_0,\gamma_1$,
\[
  \inf_{g\in\Isom^+(\R^3)}
  \Flat(T_{\gamma_0}-g_\#T_{\gamma_1})=0
\]
if and only if the curves have the same oriented image up to an
orientation-preserving Euclidean isometry.
\end{lemma}

\begin{proof}
The reverse implication is immediate.  Conversely, choose
$g_j(x)=R_jx+a_j$ such that
$g_{j\#}T_{\gamma_1}\to T_{\gamma_0}$ in flat norm.  After passing to a
subsequence, $R_j\to R$ because $SO(3)$ is compact.  The translations
$a_j$ are bounded: otherwise the compact supports of
$g_{j\#}T_{\gamma_1}$ escape every compact set, contradicting weak
convergence to the nonzero compactly supported current $T_{\gamma_0}$.
Passing to another subsequence gives $a_j\to a$ and hence
$g_{j\#}T_{\gamma_1}\to g_\#T_{\gamma_1}$, where $g(x)=Rx+a$.
Uniqueness of the flat limit gives
$g_\#T_{\gamma_1}=T_{\gamma_0}$.  Since both currents are carried with
multiplicity one by embedded oriented curves, their oriented images agree.
\end{proof}

\begin{theorem}[Swept-area metric]
\label{thm:swept-area-metric}
For every $\Lambda\ge\Rop(K)$, the swept-area cost $d_\Lambda^K$ is a
genuine extended metric on $Y_\Lambda(K)$.  Its restriction to every
admissible component is a finite metric.
\end{theorem}

\begin{proof}
Only separation remains after \cref{prop:pseudometric-formal}.  If
$d_\Lambda^K(x_0,x_1)=0$, then
$\Flat_{\mathrm{orb}}(x_0,x_1)=0$ by
\cref{lem:flat-current-lower-bound}.  By
\cref{lem:flat-orbit-separation}, the endpoint curves differ only by an
allowed reparametrization and Euclidean isometry, so $x_0=x_1$.
\end{proof}

\begin{remark}[Relation with the classical flat-norm bound]
\label{rem:classical-nondegeneracy}
Without thickness and length constraints, the inequality between swept area and flat distance, and the resulting non-degeneracy on unparametrized curves, are classical in the theory of curve metrics \cite{MichorMumford,YezziMennucci}.  The constraints only shrink the class of admissible isotopies, so they can only increase $d_\Lambda^K$; the argument above therefore applies verbatim.  The only additional points are the quotient by Euclidean isometries, handled by \cref{lem:flat-orbit-separation}, and the fact that $d^K_\Lambda$ is infinite between classes that are not admissibly connected at the level $\Lambda$.
\end{remark}

\begin{remark}[Unoriented curves]
For unoriented knot types, replace $\Diff^{1,1}_+(S^1)$ by
$\Diff^{1,1}(S^1)$ and
replace $\Flat_{\mathrm{orb}}$ by the infimum of
$\Flat(T_{\gamma_0}-\epsilon g_\#T_{\gamma_1})$ over
$g\in\Isom^+(\R^3)$ and $\epsilon\in\{\pm1\}$.  The same proof gives an
extended metric on the unoriented moduli space.
\end{remark}

\begin{proposition}[Monotonicity in the ropelength level]
\label{prop:d-lambda-monotone}
Let $\Lambda\le \Lambda'$ and suppose that $x_0,x_1\in Y_\Lambda(K)$.  Then
\[
  d_{\Lambda'}^K(x_0,x_1)\le d_{\Lambda}^K(x_0,x_1).
\]
In particular, for a fixed pair of moduli classes, the function
$\Lambda\mapsto d_\Lambda^K(x_0,x_1)$ is nonincreasing on the set of levels at which both classes occur.
\end{proposition}

\begin{proof}
Every $\Lambda$-admissible isotopy is also $\Lambda'$-admissible when $\Lambda\le\Lambda'$.  Hence the infimum defining $d_{\Lambda'}^K$ is taken over a larger class of isotopies than the infimum defining $d_{\Lambda}^K$.

\end{proof}

\section{Admissible fundamental groups and swept-area length}
\label{sec:admissible-fundamental-groups}

The point-to-point metric is only one part of the structure carried by an admissible component.  Since an admissible isotopy has a trace surface, closed admissible isotopies also define loop-cost invariants.  To avoid ambiguity caused by the symmetry quotient, the based construction in this section is made in the parametrized, unquotiented space.  This construction is different from the classical knot concordance group: it depends on the level $\Lambda$ and on the thickness constraint, and it is attached to a fixed admissible component rather than to the knot type alone.

Let $\widetilde C$ be a $\Lambda$-admissible path component of $\widetilde Y_\Lambda(K)$, and let $\gamma_*\in \widetilde C$.  Its image lies in one admissible component of $Y_\Lambda(K)$.  Changing the actual base representative changes the based group by the usual path-conjugation isomorphism.  The associated length functions need not be exactly isometric; a quantitative additive bound is proved below.

\begin{definition}[Based admissible loop]
Let $\widetilde C$ be a $\Lambda$-admissible path component of $\widetilde Y_\Lambda(K)$, and let $\gamma_*\in\widetilde C$.  A \emph{$\Lambda$-admissible loop based at $\gamma_*$} is a $\Lambda$-admissible isotopy, in the regularity class fixed in \cref{subsec:regularity-conventions},
\[
  \Gamma:S^1\times[0,1]\longrightarrow \R^3
\]
such that
\[
  \Gamma(\cdot,0)=\Gamma(\cdot,1)=\gamma_* .
\]
Thus the base curve is fixed as an actual representative in $\widetilde Y_\Lambda(K)$, not merely up to a Euclidean motion.
\end{definition}

\begin{definition}[Based admissible homotopy]
Two based $\Lambda$-admissible loops $\Gamma_0$ and $\Gamma_1$ are \emph{based-admissibly homotopic} if there is a map
\[
  H:S^1\times[0,1]\times[0,1]\longrightarrow \R^3
\]
in the corresponding loop-homotopy regularity class such that $H(\cdot,\cdot,0)=\Gamma_0$, $H(\cdot,\cdot,1)=\Gamma_1$, and, for each $u\in[0,1]$, the map
\[
  H_u(s,t)=H(s,t,u)
\]
is a $\Lambda$-admissible loop based at $\gamma_*$.  In particular,
\[
  H(\cdot,0,u)=H(\cdot,1,u)=\gamma_*
  \qquad\text{for all }u\in[0,1].
\]
Equivalently, the homotopy is a path in the space of based $\Lambda$-admissible loops.  We do not claim that arbitrary free homotopies can be converted into based ones under the thickness and length constraints.
\end{definition}

\begin{remark}
The notation $\pi_1^\Lambda$ records the analogy with a fundamental group, but the definition is deliberately tied to the stated admissible regularity class.  No identification with the ordinary fundamental group of a particular completion or Banach topology is needed in this paper.
\end{remark}

\begin{lemma}[Based admissible homotopy is an equivalence relation]
\label{lem:based-homotopy-equivalence}
Based-admissible homotopy defines an equivalence relation on the set of based $\Lambda$-admissible loops at $\gamma_*$.
\end{lemma}

\begin{proof}
Reflexivity is given by the constant homotopy in the homotopy parameter $u$.  Symmetry is obtained by replacing $u$ with $1-u$.  For transitivity, suppose $H^0$ is a based admissible homotopy from $\Gamma_0$ to $\Gamma_1$ and $H^1$ is one from $\Gamma_1$ to $\Gamma_2$.  Concatenate them in the $u$-variable.  Each intermediate value of $u$ is one of the loops appearing in $H^0$ or $H^1$, hence is still a based $\Lambda$-admissible loop.  The base curve is fixed throughout, so the resulting homotopy is again based.
\end{proof}

\begin{definition}[Admissible fundamental group]
The set of based-admissible homotopy classes of based $\Lambda$-admissible loops is denoted by
\[
  \pi_1^\Lambda(\widetilde C,\gamma_*).
\]
The product is given by concatenation in the time variable,
\[
  [\Gamma_1]\,[\Gamma_2]=[\Gamma_1*\Gamma_2],
\]
and the inverse is represented by time reversal,
\[
  \overline\Gamma(s,t)=\Gamma(s,1-t).
\]
\end{definition}

\begin{proposition}
\label{prop:admissible-fundamental-group}
The set $\pi_1^\Lambda(\widetilde C,\gamma_*)$ is a group under concatenation.
\end{proposition}

\begin{proof}
By Lemma~\ref{lem:based-homotopy-equivalence}, the quotient by based admissible homotopy is well defined.  Concatenation and time reversal preserve $\Lambda$-admissibility because their time slices are time slices of the original admissible loops.  The identity is the constant loop at $\gamma_*$, which is $\Lambda$-admissible since $\gamma_*\in\widetilde Y_\Lambda(K)$.

Associativity is represented by the usual homotopy of the time subdivision of $[0,1]$.  This homotopy changes only the speed with which the three loops are traversed and inserts or removes constant subintervals.  Thus every intermediate loop has the same set of time-slice curves as the original loops, possibly with repeated constant intervals, and remains based and $\Lambda$-admissible.

For the inverse law, let $r(t)$ be the triangular function $r(t)=2t$ for $0\le t\le1/2$ and $r(t)=2-2t$ for $1/2\le t\le1$.  The loop $\Gamma*\overline\Gamma$ is the path $t\mapsto\Gamma(\cdot,r(t))$.  Define $r_u(t)=(1-u)r(t)$ and
\[
  H(s,t,u)=\Gamma(s,r_u(t)).
\]
For each $u$, the time slices of $H_u$ are time slices of $\Gamma$, and $H(\cdot,0,u)=H(\cdot,1,u)=\gamma_*$.  Hence $H_u$ is a based $\Lambda$-admissible loop.  At $u=0$ it is $\Gamma*\overline\Gamma$, while at $u=1$ it is the constant loop at $\gamma_*$.  This proves the inverse law.
\end{proof}

\begin{remark}[Free admissible loops]
If one does not want to choose a base representative, one may instead consider closed $\Lambda$-admissible isotopies modulo free admissible homotopy.  This gives the set of free admissible loop classes, corresponding to conjugacy classes in the based group when the usual basepoint-change operations are admissible.  The free version is often the more natural object on the quotient $Y_\Lambda(K)$.  It is not, in general, a group; the group structure belongs to the based construction above.
\end{remark}

\begin{definition}[Swept-area length on the admissible fundamental group]
For $\alpha\in\pi_1^\Lambda(\widetilde C,\gamma_*)$, define
\[
  \ell_\Lambda(\alpha)=
  \inf\{A(\Gamma)\mid \Gamma \text{ is a based $\Lambda$-admissible loop representing }\alpha\}.
\]
\end{definition}

\begin{definition}[Swept-area length group]
The pair
\[
  \mathcal G_\Lambda(\widetilde C,\gamma_*)
  =
  \bigl(\pi_1^\Lambda(\widetilde C,\gamma_*),\ell_\Lambda\bigr)
\]
is called the \emph{swept-area length group} of the based admissible component
$(\widetilde C,\gamma_*)$.
\end{definition}

\begin{theorem}[Swept-area group length]
\label{thm:area-group-length}
The function
\[
  \ell_\Lambda:\pi_1^\Lambda(\widetilde C,\gamma_*)\longrightarrow[0,+\infty)
\]
satisfies
\[
  \ell_\Lambda(e)=0,
  \qquad
  \ell_\Lambda(\alpha^{-1})=\ell_\Lambda(\alpha),
\]
and
\[
  \ell_\Lambda(\alpha\beta)
  \le
  \ell_\Lambda(\alpha)+\ell_\Lambda(\beta).
\]
Thus $\ell_\Lambda$ is a group length function, possibly degenerate on nontrivial elements.
\end{theorem}

\begin{proof}
The constant loop represents the identity and has swept area zero.  Time reversal preserves swept area, so $\ell_\Lambda(\alpha^{-1})=\ell_\Lambda(\alpha)$.  For subadditivity, choose representatives $\Gamma_\alpha$ and $\Gamma_\beta$ with swept areas within $\eps$ of the two infima.  The concatenation represents $\alpha\beta$ and satisfies
\[
  A(\Gamma_\alpha*\Gamma_\beta)=A(\Gamma_\alpha)+A(\Gamma_\beta).
\]
Letting $\eps\to0$ gives the desired inequality.
\end{proof}

\begin{proposition}[Basepoint-change control]
\label{prop:basepoint-change}
Write $\ell_{\Lambda,\gamma_*}$ when the base representative needs to be
displayed explicitly.
Let $P$ be a $\Lambda$-admissible path from $\gamma_*$ to $\gamma_*'$
inside $\widetilde C$.  The usual basepoint-change isomorphism
\[
  \Phi_P:\pi_1^\Lambda(\widetilde C,\gamma_*)
  \longrightarrow
  \pi_1^\Lambda(\widetilde C,\gamma_*')
\]
is represented by $\Gamma\mapsto\overline P*\Gamma*P$ and satisfies
\[
  \bigl|\ell_{\Lambda,\gamma_*'}(\Phi_P(\alpha))
       -\ell_{\Lambda,\gamma_*}(\alpha)\bigr|
  \le 2A(P).
\]
\end{proposition}

\begin{proof}
Concatenation gives
\[
  \ell_{\Lambda,\gamma_*'}(\Phi_P(\alpha))
  \le \ell_{\Lambda,\gamma_*}(\alpha)+2A(P).
\]
Apply the same estimate to the inverse basepoint change, represented by
$\overline P$, to obtain the reverse inequality.
\end{proof}

\begin{remark}[No conjugation-invariance claim]
The axioms proved in \cref{thm:area-group-length} do not imply
conjugation invariance, and such invariance is not asserted here.  Indeed,
the direct estimate is
\[
  \ell_\Lambda(\beta\alpha\beta^{-1})
  \le \ell_\Lambda(\alpha)+2\ell_\Lambda(\beta).
\]
Free loop classes, rather than based group elements, are the natural setting
for a conjugacy-invariant spectrum.
\end{remark}

\begin{remark}[Basic scope of the swept-area length group]
The swept-area length group is used here only as the basic group-valued invariant associated with a based admissible component.  Further refinements, such as stable lengths, spectra, growth functions, or filling functions, are natural but will not be pursued in this paper.
\end{remark}

\section{Static geometry, filtered topology, and dynamic cost}
\label{sec:three-layer}

We now record a conceptual structure which will be used throughout the rest of the paper.  There are three complementary levels of geometry associated with a knot type.  The terminology is related to the author's preprints \cite{OzawaUnifiedPDensities,OzawaIdealStratum,OzawaFiniteKnotTheory}, but the definitions used below are given explicitly here and no proof in this paper depends on results from those preprints.

\begin{enumerate}[label=\textup{(\roman*)}]
  \item \emph{Static geometry}: numerical invariants of individual representatives, such as density and compression radius.
  \item \emph{Filtered topology}: the ropelength sublevel spaces $Y_{\Lambda}(K)$, their admissible components, ideal strata, and merge scales.
  \item \emph{Dynamic geometry}: costs of paths in the filtered spaces, measured here by swept area.
\end{enumerate}

The point of this section is not to prove a definitive comparison theorem, but to make precise how these three layers fit together and to isolate the estimates which would turn static invariants into lower bounds for swept-area distance.

\subsection{Static density and compression-radius invariants}

Let $D$ be a Euclidean-invariant, scale-covariant size functional on embedded curves.  In other words, $D$ is unchanged by orientation-preserving Euclidean isometries and satisfies
\[
  D(\lambda \gamma)=\lambda D(\gamma)
  \qquad
  \text{for }\lambda>0.
\]
Typical examples include the diameter, the radius of a minimal enclosing ball, or a mean-chord size functional.
Associated with $D$ are the $D$-density and the $D$-compression radius
\[
  \rho_D(\gamma)=\frac{\Len(\gamma)}{D(\gamma)},
  \qquad
  \CRad_D(\gamma)=\frac{D(\gamma)}{\Thi(\gamma)}.
\]
They satisfy the pointwise factorization
\[
  \Rop(\gamma)=\rho_D(\gamma)\,\CRad_D(\gamma).
\]
Indeed,
\[
  \rho_D(\gamma)\,\CRad_D(\gamma)
  =\frac{\Len(\gamma)}{D(\gamma)}\cdot\frac{D(\gamma)}{\Thi(\gamma)}
  =\frac{\Len(\gamma)}{\Thi(\gamma)}
  =\Rop(\gamma).
\]
Thus density and compression radius split ropelength into a length-to-size part and a size-to-thickness part.  This elementary identity is the only factorization property used in the present paper; the cited preprint is used here only for motivation and terminology.

After thickness normalization, $\Thi(\gamma)\ge 1$, the compression radius is closely related to the size functional $D(\gamma)$, while the density records how efficiently length is distributed relative to that size.  These Euclidean- and reparametrization-invariant quantities descend to functions on the filtered space $Y_{\Lambda}(K)$.

For mean-chord constructions, following \cite{OzawaUnifiedPDensities}, one may similarly take a family of size functionals $D_p(\gamma)$ and define
\[
  \rho_p(\gamma)=\frac{\Len(\gamma)}{D_p(\gamma)}.
\]
The exact definition of $D_p$ is not needed here.  What matters for the present framework is that such quantities are scale-free functions on the same ropelength-filtered spaces on which the swept-area metric is defined.

\subsection{Filtered topology}

The spaces $Y_{\Lambda}(K)$ form the filtered topological layer, in the sense of \cite{OzawaIdealStratum}.  They record which thick representatives are available at a given length level.  Their admissible components encode qualitative connectivity, while merge scales record the first value of $\Lambda$ at which two components become connected.

Thus filtered topology answers the question
\[
  \text{When are two representatives connected under a thickness and length constraint?}
\]
It does not, by itself, measure how costly such a connection is.

\subsection{Dynamic swept-area cost}

The swept-area metric is the dynamic layer.  If $x_0$ and $x_1$ lie in the same admissible component of $Y_{\Lambda}(K)$, then
\[
  d_{\Lambda}^{K}(x_0,x_1)
  =
  \inf_{\Gamma} A(\Gamma)
\]
measures the least area cost of an admissible deformation from the class $x_0$ to the class $x_1$.  In this sense, $d_{\Lambda}^{K}$ is not another static invariant of a single curve.  It is a cost assigned to paths, or equivalently a metric structure on the filtered space.

This gives the following schematic picture:
\[
  \begin{array}{ccl}
  \text{static geometry} &:& \rho_D,\ \CRad_D,\ \rho_p
  \quad \text{as functions on points},\\[2mm]
  \text{filtered topology} &:& Y_{\Lambda}(K),\ I(K),\ m(C_i,C_j)
  \quad \text{as connectivity data},\\[2mm]
  \text{dynamic geometry} &:& d_{\Lambda}^{K}
  \quad \text{as path cost}.
  \end{array}
\]

\subsection{A conditional comparison principle}

A natural way to relate the static and dynamic layers is to ask whether variation of a static invariant along an admissible isotopy can be controlled by swept area.  This is not automatic.  It depends on the regularity class and on the chosen size functional.  We therefore separate a simple conditional statement from the open problems.

Let $F:Y_\Lambda(K)\to\R$ be a real-valued function on moduli classes.  We say that $F$ is \emph{area-Lipschitz on $Y_\Lambda(K)$ with constant $L_F$} if every admissible isotopy $\Gamma$ whose endpoints represent $x_0,x_1$ satisfies
\[
  |F(x_1)-F(x_0)|\le L_F A(\Gamma).
\]
This is a strong hypothesis, but it is exactly the hypothesis needed to turn a static invariant into a lower bound for swept-area distance.

\begin{proposition}[Static variation gives a lower bound]
\label{prop:area-lipschitz-lower-bound}
Let $F$ be area-Lipschitz on $Y_\Lambda(K)$ with constant $L_F$.  Then, for any $x_0,x_1\in Y_\Lambda(K)$ lying in the same admissible component,
\[
  |F(x_1)-F(x_0)|
  \le L_F d_\Lambda^K(x_0,x_1).
\]
Equivalently,
\[
  d_\Lambda^K(x_0,x_1)
  \ge \frac{|F(x_1)-F(x_0)|}{L_F}
\]
whenever $L_F>0$.
\end{proposition}

\begin{proof}
The defining inequality holds for every admissible isotopy with endpoints in the two moduli classes.  Taking the infimum gives the result.
\end{proof}

In particular, if $F$ is one of the static quantities $\rho_D$, $\CRad_D$, or $\rho_p$, then any independent proof of area-Lipschitz control for $F$ immediately yields a nontrivial lower bound for swept-area distance.  Establishing such control for natural functionals such as diameter, enclosing radius, or mean-chord size is a separate analytic problem.  We do not claim it in full generality here.

A first nontrivial instance is obtained from the projected-area vector.  For an oriented closed curve define
\[
  \AreaVec(\gamma)
  =\frac12\int_{S^1}\gamma(s)\times\partial_s\gamma(s)\,ds
  =\frac12\int_\gamma x\times dx\in\R^3.
\]
If $n$ is an oriented unit normal and $\Pi=n^\perp$, then
\[
  \mathcal A_\Pi(\gamma)=n\cdot\AreaVec(\gamma)
\]
is the signed area of the orthogonal projection to $\Pi$.  The vector is translation invariant and transforms by
$\AreaVec(g\gamma)=R\AreaVec(\gamma)$ when $g(x)=Rx+a$.

\begin{proposition}[Projected-area vector is area-Lipschitz]
\label{prop:area-vector-lipschitz}
Let $\Gamma$ be an admissible isotopy from an oriented curve $\gamma_0$ to an oriented curve $\gamma_1$.  Then
\[
  |\AreaVec(\gamma_1)-\AreaVec(\gamma_0)|\le A(\Gamma).
\]
Consequently, for the exact-endpoint cost,
\[
  \widetilde d_\Lambda^K(\gamma_0,\gamma_1)
  \ge |\AreaVec(\gamma_1)-\AreaVec(\gamma_0)|.
\]
\end{proposition}

\begin{proof}
For a unit vector $n$, let $\omega_n$ be the constant $2$-form
$\omega_n(v,w)=n\cdot(v\times w)$.  It has comass one and is the exterior
derivative of
\[
  \alpha_n=\frac12,n\cdot(x\times dx).
\]
Stokes' theorem therefore gives
\[
  n\cdot\bigl(\AreaVec(\gamma_1)-\AreaVec(\gamma_0)\bigr)
  =\int_{S^1\times[0,1]}\Gamma^*\omega_n.
\]
Hence
\[
  \left|n\cdot\bigl(\AreaVec(\gamma_1)-\AreaVec(\gamma_0)\bigr)\right|
  \le A(\Gamma).
\]
Taking the supremum over unit $n$ proves the vector inequality, and taking the infimum over exact-endpoint isotopies proves the second assertion.
\end{proof}

\begin{corollary}[Fixed-frame projection calibration]
\label{cor:plane-calibration}
For exact endpoint representatives,
\[
  \widetilde d_\Lambda^K(\gamma_0,\gamma_1)
  \ge
  \sup_{\Pi}\,|\mathcal A_\Pi(\gamma_1)-\mathcal A_\Pi(\gamma_0)|
  =|\AreaVec(\gamma_1)-\AreaVec(\gamma_0)|,
\]
where the supremum is taken over all oriented two-planes.
\end{corollary}

\begin{corollary}[Euclidean-invariant quotient calibration]
\label{cor:quotient-area-calibration}
The quantity
\[
  \mathscr A(x)=|\AreaVec(\gamma)|,
  \qquad x=[\gamma]\in Y_\Lambda(K),
\]
is well defined.  For all $x_i=[\gamma_i]\in Y_\Lambda(K)$,
\[
  d_\Lambda^K(x_0,x_1)
  \ge
  \bigl|\,|\AreaVec(\gamma_1)|-|\AreaVec(\gamma_0)|\,\bigr|
  =|\mathscr A(x_1)-\mathscr A(x_0)|.
\]
\end{corollary}

\begin{proof}
Translation leaves $\AreaVec$ unchanged, rotation rotates it, and an orientation-preserving reparametrization leaves it unchanged; hence its norm is a function on the quotient.  For every endpoint alignment with rotational part $R$, \cref{prop:area-vector-lipschitz} gives
\[
  A(\Gamma)\ge|R\AreaVec(\gamma_1)-\AreaVec(\gamma_0)|
  \ge\bigl|\,|\AreaVec(\gamma_1)|-|\AreaVec(\gamma_0)|\,\bigr|.
\]
Taking the quotient infimum proves the result.
\end{proof}

\begin{corollary}[Positive quotient distance from projected area]
\label{cor:positive-projected-area}
If two moduli classes have different projected-area-vector norms, then their swept-area distance is positive.
\end{corollary}

\begin{remark}
The fixed-frame bound must not be applied directly to two arbitrary representatives of separate Euclidean orbits: an independent rotation of one endpoint changes the vector difference.  The norm-difference bound in \cref{cor:quotient-area-calibration} is the corresponding well-defined statement on $Y_\Lambda(K)$.  These calibrations are concrete consequences of the more general current-theoretic lower bound in \cref{lem:flat-current-lower-bound}.

The corresponding statements for diameter, minimal enclosing radius, mean-chord size, density, and compression radius require additional analytic hypotheses and are kept as open problems below.
\end{remark}

\begin{remark}[Total curvature]
A natural candidate for another static quantity is total curvature $\TC(\gamma)$, classically studied by Fenchel and Milnor \cite{Fenchel,Milnor}.  The present paper does not claim a general inequality bounding $|\TC(\gamma_1)-\TC(\gamma_0)|$ by swept area.  Such an estimate would require control of tangent variation, not only control of the parametrized trace area.  We therefore record it below as a separate problem rather than using it as a proved comparison theorem.
\end{remark}

\subsection{Compression radius as a possible obstruction}

The compression radius has a particularly direct interpretation in the unit-thickness setting.  If $\Thi(\gamma)=1$, then
\[
  \CRad_D(\gamma)=D(\gamma).
\]
For $D=\diam$ or for the minimal enclosing radius, this measures the spatial scale occupied by the representative.  If two representatives have substantially different $D$-sizes, then any isotopy between them must change a global size parameter.  Proposition~\ref{prop:area-lipschitz-lower-bound} explains the precise form of the desired conclusion: once $D$ or $\CRad_D$ is known to be area-Lipschitz on a controlled class, its variation becomes a lower bound for swept-area distance.  Without such a hypothesis, inequalities of this type remain guiding problems rather than theorems.

\begin{remark}
The three-layer structure also clarifies why swept area was not needed for defining admissible components or merge scales.  Connectivity and merge thresholds are filtered-topological notions.  Swept area becomes relevant only after one asks for a quantitative cost of a path which already exists in the filtered space.
\end{remark}

\section{Ideal strata and merge-scale refinements}
\label{sec:ideal-strata}

The swept-area metric gives a quantitative refinement of the admissible-component structure of the ropelength filtration developed in \cite{OzawaIdealStratum}.
We first recall the merge scale.

\begin{definition}[Merge scale]
Fix a level $\lambda_0\ge\Rop(K)$, and let $C_{0}$ and $C_{1}$ be admissible components of $Y_{\lambda_0}(K)$.  For $\Lambda\ge\lambda_0$, regard them as subsets of $Y_\Lambda(K)$ under the natural inclusion.
The merge scale of $C_{0}$ and $C_{1}$ is
\[
  m(C_{0},C_{1})=
  \inf\{\Lambda\ge \lambda_0\mid C_{0}\text{ and }C_{1}\text{ lie in the same admissible component of }Y_{\Lambda}(K)\}.
\]
As usual, the infimum of the empty set is $+\infty$.
\end{definition}

Thus $m(C_{0},C_{1})$ records the first length level at which the two components become connectable by a thick isotopy.
It is a qualitative threshold.
The infimum need not itself be a level at which the merge is attained.  Once the components merge, the swept-area metric gives a quantitative cost.

\begin{definition}[Swept-area merge cost]
For every $\Lambda\ge\lambda_0$, define the extended merge cost
\[
  a_{\Lambda}(C_{0},C_{1})=
  \inf\{d_{\Lambda}^{K}(x_{0},x_{1})\mid x_{0}\in C_{0},\ x_{1}\in C_{1}\}.
\]
\end{definition}

Here $C_{0}$ and $C_{1}$ may be interpreted as components at a lower level, for instance components of the ideal stratum.
The value is $+\infty$ while the two components remain disconnected and is finite after they merge.  Thus no attainment at the threshold $m(C_0,C_1)$ is assumed.

In particular, if $I(K)$ has more than one admissible component, then one obtains a two-stage invariant:
\[
  \bigl(m(C_{0},C_{1}),\ a_{\Lambda}(C_{0},C_{1})\bigr).
\]
The first coordinate records when a connection becomes possible; the second records how much area it costs at the chosen level.

When $m(C_0,C_1)<+\infty$, one may also consider the limiting cost
\[
  a_{+}(C_{0},C_{1})=\inf_{\Lambda>m(C_{0},C_{1})}a_{\Lambda}(C_{0},C_{1}).
\]
This quantity should be regarded as a swept-area refinement of the merge scale.

\begin{proposition}[Monotonicity of merge costs]\label{prop:merge-cost-monotone}
If $\lambda_0\le\Lambda\le \Lambda'$, then
\[
  a_{\Lambda'}(C_0,C_1)\le a_{\Lambda}(C_0,C_1).
\]
Consequently, $\Lambda\mapsto a_\Lambda(C_0,C_1)$ is nonincreasing as an extended-valued function.  The quantity $a_+(C_0,C_1)$ defined above is therefore the most relaxed post-merge cost over all higher ropelength levels; it need not agree with the immediate right-hand limit at the merge scale.
\end{proposition}

\begin{proof}
Every $\Lambda$-admissible isotopy is also $\Lambda'$-admissible when $\Lambda\le\Lambda'$.  Therefore the infimum defining $a_{\Lambda'}$ is taken over a larger class of paths than the one defining $a_{\Lambda}$, which gives the stated inequality.
\end{proof}

\begin{remark}[Immediate versus relaxed post-merge cost]
Because $a_\Lambda$ is monotone, the right limit exists in $[0,+\infty]$.  If one wants the cost immediately after merging, one should consider
\[
  a_{m+}(C_0,C_1)=\lim_{\epsilon\downarrow0}a_{m(C_0,C_1)+\epsilon}(C_0,C_1),
\]
The invariant $a_+$ is deliberately more relaxed: it allows the length level to increase after the merge and records the cheapest cost available at any higher level.
\end{remark}

\begin{figure}[t]
\centering
\begin{tikzpicture}[x=1cm,y=1cm, every node/.style={font=\small}]
  \draw[-{Latex[length=2mm]}, thick] (-3.3,0) -- (4.2,0) node[right] {$\Lambda$};
  \foreach \x/\lab in {-2.7/{\Rop(K)},-0.6/{m(C_0,C_1)}}
    {\draw[thick] (\x,0.08) -- (\x,-0.08); \node[below=0.12cm] at (\x,0) {$\lab$};}
  \draw[line width=1.1pt] (-2.7,-1.05) -- (-0.6,-1.05);
  \draw[line width=1.1pt] (-2.7,-1.55) -- (-0.6,-1.55);
  \draw[line width=1.1pt] (-0.6,-1.55) -- (-0.6,-1.05);
  \draw[line width=1.1pt] (-0.6,-1.3) -- (4.0,-1.3);
  \node[left] at (-2.75,-1.05) {$C_0$};
  \node[left] at (-2.75,-1.55) {$C_1$};
  \node[font=\footnotesize,gray] at (1.8,-1.55) {one admissible component};
  \draw[-{Latex[length=2mm]}, thin] (-3.3,0.3) -- (-3.3,3.2);
  \node[rotate=90,font=\footnotesize] at (-3.65,1.75) {$a_\Lambda(C_0,C_1)$};
  \draw[gray, dashed] (-2.7,2.95) -- (-0.6,2.95);
  \node[font=\footnotesize,gray,above] at (-1.65,2.97) {$+\infty$ (not yet merged)};
  \draw[line width=1.2pt] (-0.6,2.3) -- (0.3,2.3) -- (0.3,1.75) -- (1.4,1.75) .. controls (2.2,1.6) and (2.8,1.2) .. (4.0,1.08);
  \draw[fill=white,thick] (-0.6,2.3) circle (2pt);
  \draw[dashed] (-0.6,1.0) -- (4.0,1.0);
  \node[font=\footnotesize,anchor=west] at (4.05,1.0) {$a_+(C_0,C_1)$};
  \node[font=\footnotesize,anchor=south west] at (-0.55,2.33) {$a_{m+}(C_0,C_1)$};
  \draw[gray,dotted] (-0.6,0) -- (-0.6,2.95);
\end{tikzpicture}
\caption{Merge scale and swept-area merge cost, illustrated for components born at the ideal level $\lambda_0=\Rop(K)$ (schematic).  Below $m(C_0,C_1)$ the cost is $+\infty$; above it the cost is finite and nonincreasing in $\Lambda$ (\cref{prop:merge-cost-monotone}).  The immediate cost $a_{m+}$ and the relaxed cost $a_+$ need not agree.}
\label{fig:merge-cost}
\end{figure}

\section{Polygonal models and finite-dimensional approximations}
\label{sec:polygonal}

The preceding definitions may be formulated either in a smooth setting or in a polygonal setting.
Polygonal models are useful because they give finite-dimensional approximations of ropelength-filtered knot spaces, paralleling the polygonal approximation viewpoints in \cite{OzawaUnifiedPDensities,OzawaIdealStratum}.

Fix an integer $N\ge 3$ once and for all.
All polygonal paths and homotopies in this section are required to stay inside the same labelled $N$-edge polygon space; subdivision changes or moves between different values of $N$ are not part of the definition of $d_{\Lambda,N}^K$.  Polygonal isotopies are represented by absolutely continuous paths of labelled vertices, so the vertex path is differentiable almost everywhere by Rademacher's theorem.
Let
\[
  \Pol_{N}(K)
\]
denote the space of embedded labelled $N$-edge polygons of knot type $K$.
For a polygonal curve $P$, let $\Len(P)$ be its total edge length.  We use the standard polygonal thickness in the sense of Rawdon \cite{Rawdon}, compatible with the thickness theory of Litherland--Simon--Durumeric--Rawdon \cite{LSDR}.  If $v_i$ is a vertex with adjacent edge lengths $\ell_i^-$ and $\ell_i^+$ and exterior turning angle $\theta_i\in[0,\pi]$, set
\[
  \operatorname{Rad}(v_i)=
  \frac{\min\{\ell_i^-,\ell_i^+\}}{2\tan(\theta_i/2)}
\]
with the convention $\operatorname{Rad}(v_i)=+\infty$ when $\theta_i=0$.  Define
\[
  \operatorname{MinRad}(P)=\min_i \operatorname{Rad}(v_i).
\]
Let $\operatorname{dcsd}(P)$ denote the doubly-critical self-distance, namely the minimum distance $|p-q|$ over pairs of distinct points $p,q\in P$ for which the chord $pq$ is perpendicular to the incident tangent direction or tangent cone of $P$ at both endpoints.  The polygonal thickness is then
\[
  \Thi(P)=\min\{\operatorname{MinRad}(P),\,\tfrac12\operatorname{dcsd}(P)\}.
\]

Define the parametrized labelled polygonal sublevel space and its Euclidean quotient by
\[
  \widetilde Y_{\Lambda}^{(N)}(K)=
  \{P\in \Pol_{N}(K)\mid \Thi(P)\ge 1,\ \Len(P)\le \Lambda\},
  \qquad
  Y_{\Lambda}^{(N)}(K)=
  \widetilde Y_{\Lambda}^{(N)}(K)/\Isom^{+}(\R^{3}).
\]

A polygonal isotopy is represented by an absolutely continuous vertex path.  On each edge its trace is the ruled map
\[
  \Gamma:S^{1}\times[0,1]\longrightarrow \R^{3}
\]
whose time slices are polygons with $N$ edges.  If the vertex paths are piecewise linear, the parameter cylinder may be triangulated to give the equivalent piecewise-linear trace description.

\begin{definition}[Polygonal swept-area pseudometric]
For $P_{0},P_{1}\in \widetilde Y_{\Lambda}^{(N)}(K)$, let
\[
  \widetilde d_{\Lambda,N}^{K}(P_{0},P_{1})=\inf_{\Gamma} A(\Gamma),
\]
where the infimum is taken over all exact-endpoint polygonal $\Lambda$-admissible isotopies in the fixed labelled $N$-edge space.  For quotient classes $x_i=[P_i]\in Y_{\Lambda}^{(N)}(K)$, set
\[
  d_{\Lambda,N}^{K}(x_0,x_1)
  =\inf_{g\in\Isom^+(\R^3)}
  \widetilde d_{\Lambda,N}^{K}(P_0,gP_1).
\]

If $Z\subset Y_{\Lambda}^{(N)}(K)$ is a specified subspace, we write
\[
  d_{\Lambda,N}^{K,Z}(x_0,x_1)
\]
for the restricted pseudometric obtained by taking the infimum only over polygonal admissible paths whose quotient classes stay in $Z$.
\end{definition}

The same proof as in the smooth setting gives the following.

\begin{proposition}
For each $N$ and $\Lambda$, the function $d_{\Lambda,N}^{K}$ is an extended pseudometric on $Y_{\Lambda}^{(N)}(K)$ and is finite on each polygonal admissible component.
\end{proposition}

We next record a quantitative statement that retains the vertex labels.  The current lower bound in \cref{lem:flat-current-lower-bound} separates unparametrized polygonal images, but it deliberately forgets label data.  The following finite-dimensional estimate separates labelled configurations.

\begin{definition}[Uniformly non-collinear polygonal stratum]
Fix constants $\delta>0$ and $0<\eta<\pi/2$.  Let
\[
  Y_{\Lambda,N}^{K}(\delta,\eta)\subset Y_{\Lambda}^{(N)}(K)
\]
be the subspace represented by labelled polygons whose edge lengths are at least $\delta$ and whose exterior angles lie in the interval $[\eta,\pi-\eta]$ at every vertex.
\end{definition}

\begin{lemma}[Area formula for labelled polygonal paths]
\label{lem:polygonal-area-formula}
Let $V(t)=(v_0(t),\ldots,v_{N-1}(t))$ be an absolutely continuous path of labelled polygons, and set $e_i(t)=v_{i+1}(t)-v_i(t)$, with indices modulo $N$.  Let $\Gamma$ be the corresponding piecewise-linear trace, parametrized on the $i$-th edge by
\[
  \Gamma_i(u,t)=v_i(t)+u e_i(t),\qquad 0\le u\le1 .
\]
Then, counting multiplicity and ignoring the vertex traces of measure zero,
\[
  A(\Gamma)=\int_0^1 \sum_{i=0}^{N-1}\int_0^1
  \left|e_i(t)\times\bigl((1-u)\dot v_i(t)+u\dot v_{i+1}(t)\bigr)\right|\,du\,dt .
\]
\end{lemma}

\begin{proof}
For almost every $t$, the derivatives of the vertex functions exist.  On the $i$-th ruled face,
\[
  \partial_u\Gamma_i=e_i(t),\qquad
  \partial_t\Gamma_i=(1-u)\dot v_i(t)+u\dot v_{i+1}(t).
\]
The area formula for Lipschitz maps \cite[Section~3.2.3]{Federer} gives the stated expression for the parametrized area of each face.  Summing over the finitely many faces gives the formula; possible overlaps of faces are counted with multiplicity, as in the definition of $A(\Gamma)$.
\end{proof}

\begin{theorem}[Orbit-distance control on a polygonal stratum]
\label{thm:polygonal-metric}
There is a constant $c=c(N,\Lambda,\delta,\eta)>0$ such that, for all $x_i=[V_i]$ in the same restricted path component of $Y_{\Lambda,N}^{K}(\delta,\eta)$,
\[
  d_{\Lambda,N}^{K,Y_{\Lambda,N}^{K}(\delta,\eta)}(x_0,x_1)
  \ge c\,\dist_{\mathrm{orb}}(x_0,x_1),
\]
where
\[
  \dist_{\mathrm{orb}}([V_0],[V_1])
  =\inf_{g\in\Isom^+(\R^3)}\|V_0-gV_1\|
\]
is the Euclidean orbit distance of the labelled vertex vectors.  In particular, the restricted swept-area pseudometric is a genuine metric.
\end{theorem}

\begin{proof}
Write a polygon as a vertex vector
\[
  V=(v_0,\ldots,v_{N-1})\in (\R^3)^N.
\]
For a tangent vector $W=(w_0,\ldots,w_{N-1})$, set $e_i=v_{i+1}-v_i$ and define the infinitesimal swept-area seminorm
\[
  \mathfrak a_V(W)=
  \sum_{i=0}^{N-1}\int_0^1
  \left|e_i\times\bigl((1-u)w_i+u w_{i+1}\bigr)\right|\,du.
\]
By Lemma~\ref{lem:polygonal-area-formula}, every absolutely continuous polygonal path $V(t)$ has trace area
\[
  A(\Gamma)=\int_0^1 \mathfrak a_{V(t)}(\dot V(t))\,dt .
\]
If $\mathfrak a_V(W)=0$, then for each edge the vector
$(1-u)w_i+u w_{i+1}$ is parallel to $e_i$ for almost every $u$.  In
particular, $w_i$ is parallel to both adjacent edge directions $e_{i-1}$
and $e_i$.  Since the exterior angle at every vertex lies in
$[\eta,\pi-\eta]$, these two directions are not parallel.  Hence
$w_i=0$ for every $i$, so $\mathfrak a_V$ is a norm on the full labelled vertex space.

After translating $v_0$ to the origin, the set of possible edge data is compact: every edge length lies in $[\delta,\Lambda]$, the total length is at most $\Lambda$, and the angle bounds are closed.  The value of $\mathfrak a_V(W)$ is independent of a common translation of $V$ and depends continuously on $(V,W)$.  Compactness of the normalized edge data times the unit sphere in $(\R^3)^N$, together with the trivial-kernel argument above, therefore gives a constant $c=c(N,\Lambda,\delta,\eta)>0$ such that
\[
  \mathfrak a_V(W)\ge c\|W\|
\]
for all labelled velocity vectors $W$.  If $V(t)$ is a polygonal admissible path whose entire quotient image remains in the stratum, then
\[
  A(\Gamma)= \int_0^1 \mathfrak a_{V(t)}(\dot V(t))\,dt
  \ge c\int_0^1\|\dot V(t)\|\,dt
  \ge c\|V(1)-V(0)\|.
\]
For a path used in the quotient cost, the terminal vector has the form $gV_1$.  Hence
\[
  A(\Gamma)\ge c\|gV_1-V_0\|.
\]
Taking the infimum first over paths and then over $g$ gives the asserted orbit-distance estimate.  The orbit distance is positive for distinct labelled configurations modulo Euclidean isometry: the optimal translation is bounded and the rotational group is compact.  This proves non-degeneracy.
\end{proof}

\begin{remark}
The angle bounds in \cref{thm:polygonal-metric} remove tangential degeneracies coming from parallel adjacent edges.  The theorem is a fixed-$N$ statement for labelled polygons.  If subdivision or the number of vertices is allowed to change, different labelled descriptions of the same geometric curve should either be identified from the outset or treated by an explicit zero-label-cost quotient.
\end{remark}

\begin{remark}
The relation between $d_{\Lambda,N}^{K}$ and $d_{\Lambda}^{K}$ as $N\to\infty$ is a natural approximation problem.
A full convergence theorem requires compatibility between polygonal thickness, length, and the swept-area functional.
We leave the most general form of this question open, but the polygonal model provides a controlled finite-dimensional framework for computations and examples.
\end{remark}

\section{Swept-area weighted diagram graphs}
\label{sec:weighted-reidemeister}

We finally explain the diagrammatic counterpart of the swept-area construction.  This section gives a metric counterpart of the diagram image of the ropelength-filtered Reidemeister graphs of \cite{OzawaFiniteKnotTheory}.  Unlike the lifted graphs there, the auxiliary graph below has diagrams, not fibre components, as vertices.
Fix a generic projection direction $u\in S^{2}$.
The construction is made in the framed space $\widetilde Y_\Lambda(K)$: a Euclidean moduli class does not determine a diagram for a fixed ambient direction $u$.  A representative $\gamma\in\widetilde Y_{\Lambda}(K)$ determines a knot diagram $D_{u}(\gamma)$ whenever the projection is generic.

For the purposes of this paper, define the auxiliary ropelength-filtered diagram graph
\[
  \mathcal W_{\Lambda,u}(K)
\]
as follows.  A vertex is an equivalence class of generic diagrams $D_u(\gamma)$ obtained from representatives $\gamma\in\widetilde Y_\Lambda(K)$, where diagrams are taken up to planar isotopy preserving crossing data.  An edge is not merely a pair of vertices; it is an equivalence class of lifted one-move transitions.  More precisely, an edge from $D$ to $D'$ is represented by a triple $(\gamma_0,\Gamma,\gamma_1)$ such that $\gamma_0,\gamma_1\in\widetilde Y_\Lambda(K)$ project to $D,D'$, respectively, and $\Gamma$ is a $\Lambda$-admissible isotopy for which the projected diagram changes by exactly one Reidemeister move and is generic except at the single corresponding critical time.  Two such triples represent the same edge if they differ by planar isotopy of the endpoint diagrams and by a one-parameter family $(\gamma_0^v,\Gamma^v,\gamma_1^v)$, $v\in[0,1]$, of $\Lambda$-admissible lifted one-move transitions of the same Reidemeister type, with the projected diagrams varying by planar isotopy at the endpoints.  Thus the homotopy in the edge relation is itself required to stay inside the class of lifted one-move transitions.  This auxiliary graph collapses each regular diagram fibre. It differs from the fibre-component lifted graph in \cite{OzawaFiniteKnotTheory}; in particular, it can lose positive motion costs within a fibre.

The swept-area functional gives weights to such lifted transitions.

\begin{definition}[Swept-area weight of a lifted transition]
Let $e$ be an edge of $\mathcal W_{\Lambda,u}(K)$ connecting diagrams $D$ and $D'$.
Define its swept-area weight by
\[
  w_{\Lambda,u}(e)=\inf_{\Gamma} A(\Gamma),
\]
where the infimum is taken over all $\Lambda$-admissible isotopies realizing the corresponding lifted Reidemeister transition from $D$ to $D'$.
\end{definition}

This turns $\mathcal W_{\Lambda,u}(K)$ into a weighted graph, as schematically shown in Figure~\ref{fig:weighted-reidemeister}.

\begin{remark}[Dependence on choices]
The graph $\mathcal W_{\Lambda,u}(K)$ is a framed auxiliary object depending on the projection direction $u$ and on the chosen equivalence relation on lifted transitions.  The weight of an edge is nevertheless well-defined by the preceding definition, because the edge itself includes the lifted transition class and the weight is the infimum over all admissible realizations in that class.  This paper does not claim independence of $u$ or invariance under changing the ambient Euclidean frame.  A quotient-level invariant would require an additional infimum over frames.
\end{remark}
The associated path pseudometric is defined by
\[
  d_{\Lambda,u}^{\operatorname{diag}}(D,D')=
  \inf\left\{\sum_{e\in p}w_{\Lambda,u}(e)\mid p\text{ is a path from }D\text{ to }D'\right\}.
\]

For exact endpoint representatives with generic endpoint diagrams, let
$\mathscr I_{\Lambda,u}^{\mathrm{gen}}(\gamma_0,\gamma_1)$ be the set of
$\Lambda$-admissible isotopies whose $u$-projections are generic except at
finitely many Reidemeister times.  Define
\[
  \widetilde d_{\Lambda,u}^{K,\mathrm{gen}}(\gamma_0,\gamma_1)
  =\inf_{\Gamma\in
  \mathscr I_{\Lambda,u}^{\mathrm{gen}}(\gamma_0,\gamma_1)}A(\Gamma).
\]
The value is $+\infty$ if no such isotopy exists.

\begin{theorem}[Diagrammatic distance and generic geometric cost]
\label{thm:diagrammatic-bound}
Let $\gamma_0,\gamma_1\in\widetilde Y_\Lambda(K)$ have generic endpoint diagrams.  Then
\[
  d_{\Lambda,u}^{\operatorname{diag}}(D_u(\gamma_0),D_u(\gamma_1))
  \le
  \widetilde d_{\Lambda,u}^{K,\mathrm{gen}}(\gamma_0,\gamma_1).
\]
\end{theorem}

\begin{proof}
Let $\Gamma$ be an isotopy in the class defining the right-hand side, and choose pairwise disjoint time intervals, one around each Reidemeister critical time, small enough that each chosen interval contains exactly one Reidemeister critical time and the projected family realizes exactly one Reidemeister move on that interval.  The restriction of $\Gamma$ to such an interval represents a lifted one-move transition, hence the corresponding edge weight is at most the swept area of that restricted trace.  The complementary generic intervals may have positive geometric swept area, but they do not change the diagram and are not recorded as Reidemeister edges.  The sum of the edge weights is at most the sum of the swept areas of the selected intervals, and this latter sum is at most $A(\Gamma)$.  These edges form a path from $D_u(\gamma_0)$ to $D_u(\gamma_1)$.  Taking the infimum over the stated generic class proves the inequality.
\end{proof}

\begin{corollary}[Conditional comparison with unrestricted cost]
\label{cor:conditional-diagrammatic-bound}
If the pair $(\gamma_0,\gamma_1)$ has the generic-approximation property
\[
  \widetilde d_{\Lambda,u}^{K,\mathrm{gen}}(\gamma_0,\gamma_1)
  =\widetilde d_\Lambda^K(\gamma_0,\gamma_1),
\]
then
\[
  d_{\Lambda,u}^{\operatorname{diag}}(D_u(\gamma_0),D_u(\gamma_1))
  \le\widetilde d_\Lambda^K(\gamma_0,\gamma_1).
\]
\end{corollary}

\begin{remark}
The distinction between generic and unrestricted cost is necessary at a closed constraint level.  Standard transversality suggests generic approximation when an isotopy has uniform slack in both thickness and length, but a perturbation of a path touching $\Thi=1$ or $\Len=\Lambda$ need not remain admissible.  The weighted diagrammatic graph can also underestimate geometric cost because it records only the selected Reidemeister intervals and ignores motion during diagrammatically constant intervals.
\end{remark}

\begin{example}[Calibration of an R1 transition]
For a fixed lifted Reidemeister I realization with endpoints $\gamma_0,\gamma_1$, every exact-endpoint realizing isotopy satisfies
\[
  A(\Gamma)\ge
  |\mathcal A_\Pi(\gamma_1)-\mathcal A_\Pi(\gamma_0)|
\]
for every oriented plane $\Pi$, by \cref{cor:plane-calibration}.  Therefore an edge class satisfies the rigorous bound
\[
  w_{\Lambda,u}(e)\ge
  \inf_{(\gamma_0,\Gamma,\gamma_1)\in e}
  |\mathcal A_\Pi(\gamma_1)-\mathcal A_\Pi(\gamma_0)|.
\]
The infimum is essential because the projected loop area need not be constant over all realizations of an edge class.  This is a lower bound. Contracting an isolated round loop to zero radius violates the positive thickness constraint and does not establish equality for a thick Reidemeister-I movie.
\end{example}

\begin{remark}[Why diagrammatic distance may underestimate]
If a curve is moved in space while its projection diagram is kept fixed, the geometric swept area can be positive whereas the diagrammatic path has no Reidemeister edge.  Thus the weighted graph measures transition cost rather than all ambient geometric motion.
\end{remark}

\begin{figure}[t]
\centering
\begin{tikzpicture}[node distance=2.2cm, every node/.style={font=\small}, v/.style={draw, circle, minimum size=8mm, fill=gray!8}]
  \node[v] (D1) {$D_1$};
  \node[v, right=of D1] (D2) {$D_2$};
  \node[v, right=of D2] (D3) {$D_3$};
  \node[v, below=1.35cm of D2] (D4) {$D_4$};
  \draw[thick] (D1) -- node[above] {$w(e_{12})$} (D2);
  \draw[thick] (D2) -- node[above] {$w(e_{23})$} (D3);
  \draw[thick] (D1) -- node[below left] {$w(e_{14})$} (D4);
  \draw[thick] (D4) -- node[below right] {$w(e_{43})$} (D3);
  \node[align=center, above=0.85cm of D2] {Reidemeister transitions realized by admissible isotopies,\\ weighted by minimal swept area};
  \node[align=center, below=0.25cm of D4] {$d_{\Lambda,u}^{\operatorname{diag}}(D_i,D_j)$ is the shortest weighted-path cost};
\end{tikzpicture}
\caption{A swept-area weighted diagram graph $\mathcal W_{\Lambda,u}(K)$ (schematic).}
\label{fig:weighted-reidemeister}
\end{figure}

The diagrammatic distance $d_{\Lambda,u}^{\operatorname{diag}}$ is a framed combinatorial shadow of the generic exact-endpoint cost $\widetilde d_{\Lambda,u}^{K,\mathrm{gen}}$.
It records not only the number or type of Reidemeister moves needed to pass between two diagrams, but also the least swept-area cost of realizing those moves by admissible thick deformations.

This suggests a two-level persistence structure:
\[
  \text{ropelength level}
  \quad+\quad
  \text{swept-area cost}.
\]
The first parameter controls which deformations are allowed; the second measures their geometric cost.


\section{Examples and basic estimates}
\label{sec:examples}

In this section we record elementary examples and estimates.  These examples complement the abstract current lower bound with exact computations.

\begin{example}[Translation before quotienting]
\label{ex:translation-circle}
Let $C(s)=(\cos s,\sin s,0)$ be the unit round circle and translate it by distance $h\ge0$ in the $z$-direction:
\[
  \Gamma(s,t)=(\cos s,\sin s,th),
  \qquad 0\le t\le1.
\]
Then
\[
  \partial_s\Gamma=(-\sin s,\cos s,0),
  \qquad
  \partial_t\Gamma=(0,0,h),
\]
and hence
\[
  |\partial_s\Gamma\times\partial_t\Gamma|=h.
\]
Therefore
\[
  A(\Gamma)=\int_0^1\int_0^{2\pi} h\,ds\,dt=2\pi h.
\]
This particular exact-endpoint path has positive swept area.  In the quotient space $Y_\Lambda(U)$, however, the two circles represent the same point and their quotient distance is zero, since the endpoint-alignment infimum may choose identical representatives.  This illustrates the distinction between $\widetilde d_\Lambda^U$ and $d_\Lambda^U$.
\end{example}

\begin{example}[Radial expansion of a round unknot]
\label{ex:radial-expansion}
Let $1\le R$ and consider the radial isotopy of round circles
\[
  \Gamma(s,t)=r(t)(\cos s,\sin s,0),
  \qquad r(0)=1,\quad r(1)=R,
\]
with $r(t)$ nondecreasing.  Then
\[
  |\partial_s\Gamma\times\partial_t\Gamma|=r(t)r'(t),
\]
and therefore
\[
  A(\Gamma)=2\pi\int_0^1 r(t)r'(t)\,dt
  =\pi(R^2-1).
\]
The time slice at radius $r$ has thickness $r$ and length $2\pi r$.  Hence this isotopy is $\Lambda$-admissible whenever $\Lambda\ge2\pi R$.  Consequently
\[
  d_\Lambda^U(C_1,C_R)\le \pi(R^2-1),
  \qquad \Lambda\ge2\pi R,
\]
where $C_r$ denotes the moduli class of the oriented round circle of radius $r$.  The next proposition shows that this upper bound is sharp.
\end{example}

\begin{proposition}[Exact distance between concentric round unknots]
\label{prop:round-circle-exact}
\leavevmode\par\noindent
Let $C_r$ denote the moduli class represented by $s\mapsto r(\cos s,\sin s,0)$ with the standard orientation.  Let $1\le R$ and assume $\Lambda\ge2\pi R$.  Then
\[
  d_\Lambda^U(C_1,C_R)=\pi(R^2-1).
\]
\end{proposition}

\begin{proof}
The radial isotopy in \cref{ex:radial-expansion} gives the upper bound.  The projected-area vectors of the two classes have norms $\pi$ and $\pi R^2$, respectively.  Hence \cref{cor:quotient-area-calibration} gives
\[
  d_\Lambda^U(C_1,C_R)\ge\pi(R^2-1).
\]
Combining the two bounds proves the equality.  This argument applies to every independent Euclidean alignment of the endpoints, not only to the displayed concentric representatives.
\end{proof}

\begin{remark}
The proof uses the Euclidean-invariant norm of the projected-area vector.  This invariant formulation is necessary because the distance is defined on Euclidean moduli classes.
\end{remark}

\begin{lemma}[Thickness of a centered ellipse]
\label{lem:ellipse-thickness}
Let $E_r$ be the planar ellipse with semi-axes $ra\ge rb>0$.  Its minimum radius of curvature is $r b^2/a$, and every doubly-critical chord has length at least $2rb$.  Consequently
\[
  \Thi(E_r)=\min\{r b^2/a,rb\}=r b^2/a.
\]
\end{lemma}

\begin{proof}
The curvature radius of the ellipse is minimized at the endpoints of the major axis and equals $r b^2/a$.  For the self-distance term, observe that a doubly-critical chord of a strictly convex plane oval is perpendicular to the tangent lines at both endpoints.  Hence the two tangent lines are parallel support lines for the oval, and the chord length is the width of the oval in the corresponding normal direction.  The minimum width of the ellipse is $2rb$, attained in the minor-axis direction.  Thus the doubly-critical self-distance is at least $2rb$, and its half is at least $rb$.  The standard thickness formula for $C^{1,1}$ curves, as the minimum of curvature radius and one half of the doubly-critical self-distance, gives the result.
\end{proof}

\begin{proposition}[Homothetic planar ellipses]
\label{prop:ellipse-exact}
Let $a\ge b>0$ and let $E_r$ denote the moduli class represented by
\[
  E_r(s)=(ra\cos s,rb\sin s,0),\qquad 0\le s\le2\pi.
\]
Let $1\le R$.  Assume that
\[
  \frac{b^2}{a}\ge 1
  \qquad\text{and}\qquad
  \Lambda\ge R\,\Len(E_1).
\]
Then the homothetic planar isotopy from $E_1$ to $E_R$ is $\Lambda$-admissible, and
\[
  d_\Lambda^U(E_1,E_R)=\pi ab(R^2-1).
\]
More generally, the same equality holds whenever the homothetic planar isotopy stays inside $Y_\Lambda(U)$.
\end{proposition}

\begin{proof}
By \cref{lem:ellipse-thickness},
\[
  \Thi(E_r)=r b^2/a.
\]  The hypothesis $b^2/a\ge1$ therefore implies
$\Thi(E_r)\ge1$ for every $r\in[1,R]$.  Also
$\Len(E_r)=r\Len(E_1)\le R\Len(E_1)\le\Lambda$.  Thus the homothetic
isotopy is $\Lambda$-admissible.

The homothetic isotopy $\Gamma(s,t)=r(t)(a\cos s,b\sin s,0)$ has
\[
  |\partial_s\Gamma\times\partial_t\Gamma|=ab\,r(t)r'(t)
\]
when $r$ is nondecreasing, and hence gives the upper bound
\[
  A(\Gamma)=\pi ab(R^2-1).
\]
The area-vector norms of $E_1$ and $E_R$ are $\pi ab$ and $\pi abR^2$.  Hence \cref{cor:quotient-area-calibration} gives
\[
  d_\Lambda^U(E_1,E_R)\ge\pi ab(R^2-1).
\]
Together with the homothetic upper bound, this proves the equality.
\end{proof}

\begin{remark}
The explicit hypotheses in \cref{prop:ellipse-exact} are only a convenient sufficient condition.  If $b^2/a<1$, one may instead start from a larger scale $r_0\ge a/b^2$ and compare $E_{r_0}$ with $E_{Rr_0}$ under the analogous length bound.  The point of the example is that, whenever the planar homothetic path stays inside the selected ropelength window, the same projected-area calibration gives an exact swept-area distance.
\end{remark}

\begin{lemma}[Constant curvature convex $C^{1,1}$ curves]
\label{lem:c11-unit-circle}
Let $\gamma$ be a closed convex plane $C^{1,1}$ curve parametrized by arclength.  If its curvature satisfies $\kappa=1$ almost everywhere and its length is $2\pi$, then $\gamma$ is a unit round circle up to orientation-preserving Euclidean isometry.
\end{lemma}

\begin{proof}
For a convex plane $C^{1,1}$ curve, the unit tangent has an absolutely continuous tangent angle $\theta(s)$ with $\theta'(s)=\kappa(s)$ almost everywhere.  Since $\kappa=1$ a.e., after adding a constant we have $\theta(s)=s+\theta_0$ a.e.  Hence $\gamma'(s)=(\cos(s+\theta_0),\sin(s+\theta_0))$ a.e.  Integrating over an interval of length $2\pi$ gives a translate and rotation of the unit circle.
\end{proof}

The following fact is standard: it is the equality case of Fenchel's theorem combined with the curvature bound of thickness, and the same argument appears in \cite{OzawaFiniteKnotTheory}.  We include it only to make the swept-area metric on the ideal stratum of the unknot explicit.

\begin{proposition}[The ideal unknot]
\label{prop:ideal-unknot-rigidity}
Let $U$ be the unknot.  In the $C^{1,1}$ unit-thickness category,
\[
  I(U)=Y_{2\pi}(U)
\]
consists of a single point, represented by the unit round circle modulo orientation-preserving reparametrizations and Euclidean isometries.  In particular, the ideal stratum of the unknot is rigid.
\end{proposition}

\begin{proof}
The unit circle shows $\Rop(U)\le2\pi$.  Conversely, scale any representative of $U$ to thickness one.  Its curvature is at most one almost everywhere, while Fenchel's theorem gives total curvature at least $2\pi$; hence its length is at least $2\pi$.  Thus $\Rop(U)=2\pi$.

Now choose a representative $\gamma$ of a class in $Y_{2\pi}(U)$.  Then $\Thi(\gamma)\ge1$ and $\Len(\gamma)\le2\pi$.  For a $C^{1,1}$ curve of thickness at least one, the curvature satisfies $\kappa\le1$ almost everywhere; this follows from the curvature-radius part of the standard thickness characterization \cite{LSDR,GonzalezMaddocks}.  Hence its total curvature satisfies
\[
  \operatorname{TC}(\gamma)=\int_\gamma \kappa\,ds\le \Len(\gamma)\le2\pi.
\]
By Fenchel's theorem \cite{Fenchel}, every closed space curve satisfies $\operatorname{TC}(\gamma)\ge2\pi$.  Therefore equality holds throughout:
\[
  \operatorname{TC}(\gamma)=\Len(\gamma)=2\pi,
  \qquad \kappa=1 \text{ a.e.}
\]
The equality case in Fenchel's theorem \cite{Fenchel} implies that $\gamma$
is a convex plane curve.  Lemma~\ref{lem:c11-unit-circle}, together with
$\kappa=1$ a.e. and $\Len(\gamma)=2\pi$, now shows that $\gamma$ is a unit
round circle.  Thus all elements of $Y_{2\pi}(U)$ agree in the stated moduli
space.
\end{proof}

\begin{proposition}[A crude displacement estimate]
Let $\Gamma$ be a smooth isotopy in any fixed regular parametrization.  Then
\[
  A(\Gamma)
  \le
  \int_{0}^{1}\Len(\Gamma_{t})\,\|\partial_{t}\Gamma(\cdot,t)\|_{L^{\infty}}\,dt.
\]
In particular, if $\Len(\Gamma_{t})\le \Lambda$ and $\|\partial_{t}\Gamma(\cdot,t)\|_{L^{\infty}}\le V(t)$, then
\[
  A(\Gamma)\le \Lambda\int_{0}^{1}V(t)\,dt.
\]
\end{proposition}

\begin{proof}
The pointwise cross-product inequality gives
\[
  |\partial_{s}\Gamma\times\partial_{t}\Gamma|
  \le |\partial_{s}\Gamma|\,|\partial_{t}\Gamma|
  .
\]
For fixed $t$, integrate in $s$ and use
$\int_{S^1}|\partial_s\Gamma|\,ds=\Len(\Gamma_t)$, followed by the $L^\infty$ bound on $|\partial_t\Gamma|$.  Integrating in $t$ gives the estimate.
\end{proof}

\begin{remark}
The examples above include a complete computation of $d_\Lambda^K$ for concentric round unknot representatives.  Beyond such calibrated examples, exact computation of $d_\Lambda^K$ is expected to be difficult and is one of the motivations for the estimates and finite-dimensional models discussed in this paper.
\end{remark}

\section{Budget--cost profiles and compatible transitions}
\label{sec:budget-cost}
For fixed endpoint classes $x,y$, define
\[
 \mathcal U_{x,y}=\{(a,q):\text{there is an admissible }H:x\leadsto y,
 \ \sup_t\Len(H_t)\le a,\ A(H)\le q\}.
\]
The thickness constraint is fixed throughout. This is an upper set, and
concatenation combines budgets by maximum and costs by addition.
The section infimum is $d_a(x,y)=\inf\{q:(a,q)\in\mathcal U_{x,y}\}$. Since the thickness constraint is normalized, the budget $a$ is a ropelength level, and $d_a$ coincides with the swept-area cost $d^K_a$ of \cref{sec:metric}.
Every jointly Lipschitz trace on the compact parameter cylinder has finite
area. Therefore the admissible merge threshold satisfies
\[
 \mu^{\rm adm}(x,y)=\inf\{a:d_a(x,y)<\infty\}.
\]
This identity asserts no attainment. Right-relaxed interpolation can
identify this threshold with a topological path threshold when the model,
orientations and symmetry conventions coincide. It does not imply that
minimal costs converge under a change of path regularity.

\begin{proposition}[Transporting budget and cost]
Suppose a whole-path simulation $F$ satisfies
$b_Y(FH)\le b_X(H)+\varepsilon$ and
$c_Y(FH)\le\lambda c_X(H)+\eta$, with $\lambda>0$. Then
\[
 d^Y_{a+\varepsilon}(Fx,Fy)\le\lambda d^X_a(x,y)+\eta.
\]
If a reverse simulation has bounds $b_X(GJ)\le b_Y(J)+\varepsilon$,
$c_X(GJ)\le\lambda'c_Y(J)+\eta'$, and connectors from $x$ to $GFx$ and
$y$ to $GFy$ of cost at most $\kappa_x,\kappa_y$ and budgets at most
$r_X(x)+2\varepsilon,r_X(y)+2\varepsilon$, then, for $a$ at least the
endpoint resources,
\[
 d^X_{a+2\varepsilon}(x,y)
 \le\lambda'd^Y_{a+\varepsilon}(Fx,Fy)+\eta'+\kappa_x+\kappa_y.
\]
\end{proposition}
\begin{proof}
Apply the forward estimate to paths approaching the infimum. For the
reverse estimate, simulate near-minimizing $Y$ paths and concatenate the
two connectors; budgets take a maximum and costs add. Then take infima.
\end{proof}
A per-edge additive error $\eta$ yields $n\eta$ along $n$ edges. A bound
uniform over whole paths cannot be inferred from per-edge estimates without
an additional bound on the number of edges.

\subsection{An exact-endpoint network}
For a fixed frame and budget, take actual regular spatial curves as vertices.
Include edges for every admissible regular-fibre segment and every
one-Reidemeister-event movie, with their exact spatial endpoints. Give each
edge class the infimum of trace area of its realizations with those endpoints.
Then the infimum of sums along finite paths is precisely the infimal generic
geometric cost for the same endpoints. To prove one inequality, cut a generic
isotopy into regular and one-event intervals; the sum of edge infima is at
most its area. For the other, choose realizations within an arbitrarily
small total error of the finitely many edge infima and concatenate them.
Endpoints match exactly, so no unaccounted joining motion is needed.
The generic class here is defined using finitely many events and is closed
under these finite concatenations.

Thus a sound finite subnetwork yields an upper bound on generic geometric
cost. The collapsed diagram graph of the preceding section yields a lower
bound. Identifying all lifts of a diagram can erase positive cost: two
concentric circles of radii $1\le R_0<R_1$ have the same diagram, but their
geometric distance at budget at least $2\pi R_1$ is
$\pi(R_1^2-R_0^2)$ by radial expansion and the area-vector calibration.
Neither graph bound identifies the unrestricted metric without the separate
cost-preserving generic approximation hypothesis.

\subsection{Scaling and infimal boundaries}
A dilation by $s$ preserves ropelength and multiplies trace area by $s^2$.
Accordingly swept-area distance is not a metric on a similarity quotient
without an additional normalization. The upper set $\mathcal U_{x,y}$ can
have unattained boundary points: compactness of a space of individual
configurations does not supply compactness of traces with the regularity
needed for a minimizing movie. Existence of a Pareto-optimal movie remains
an analytic issue separate from the budget-threshold identity.

\section{Questions and further directions}
\label{sec:questions}

We close with a short list of problems that mark the boundary between the metric framework established here and the analytic theory that remains to be developed.

\begin{question}[Comparison with the flat-orbit metric]
Theorem~\ref{thm:swept-area-metric} gives
\[
  d_\Lambda^K\ge\Flat_{\mathrm{orb}}.
\]
On which compact regularity classes does a converse estimate, possibly nonlinear or local, hold?  Does swept-area convergence together with the thickness and length bounds force convergence in a stronger topology on unparametrized curves?
\end{question}

\begin{question}[Existence of minimizers]
Under what compactness and regularity hypotheses does the infimum defining
\[
  d_{\Lambda}^{K}(x_{0},x_{1})
\]
admit a minimizing admissible isotopy?  More generally, what is the correct weak compactness theory for swept-area minimizing deformations under simultaneous thickness and length constraints?
Without the thickness and length constraints, existence of minimizing paths for related curve metrics is discussed by Yezzi and Mennucci \cite{YezziMennucci}, and minimum-area problems of this type are closely related to the Plateau problem \cite{Douglas,Rado,Lawson}.  Under the constraints $\Thi\ge1$ and $\Len\le\Lambda$, time slices already lie in a $C^{1,1}$-compact family; the difficulty is to control the time direction and to keep minimizing sequences admissible in the limit.
\end{question}

\begin{question}[Smooth--polygonal approximation]
Can the polygonal metrics $d_{\Lambda,N}^{K}$ be shown to converge, after an appropriate choice of subdivisions and compactness hypotheses, to the $C^{1,1}$ swept-area metric $d_{\Lambda}^{K}$ as $N\to\infty$?  How should labelled vertices be forgotten so that the finite-dimensional orbit estimates are compatible with the unparametrized limit?
\end{question}

\begin{question}[Static invariants and swept area]
Let $D$ be a Euclidean-invariant, scale-covariant size functional.  Can variations of the density $\rho_D=\Len/D$, the compression radius $\CRad_D=D/\Thi$, diameter, total curvature, or mean-chord densities be bounded in terms of the swept-area distance on a fixed ropelength sublevel component?  Conversely, can differences of such static quantities give robust lower bounds for swept-area distance?
\end{question}

\begin{question}[Diagrammatic approximation]
For a fixed generic projection direction $u$, how closely does the swept-area weighted diagrammatic distance
\[
  d_{\Lambda,u}^{\operatorname{diag}}
\]
approximate the framed generic cost $\widetilde d_{\Lambda,u}^{K,\mathrm{gen}}$?  Under which slack or right-relaxation hypotheses does the generic-approximation property in \cref{cor:conditional-diagrammatic-bound} hold, and how should one optimize over ambient frames to obtain a quotient-level invariant?
\end{question}

These questions suggest that swept-area metrics refine ropelength-filtered deformation theory.  They also link constrained spaces of embeddings, integral currents, finite-dimensional models, and diagrammatic persistence.

\section*{Declaration of generative AI and AI-assisted technologies}

During the preparation of this work, the author used ChatGPT (OpenAI) and Claude (Anthropic) for
literature organization, language and exposition editing, and preliminary
consistency checking of definitions and proofs.  The author reviewed and
edited the output as needed and takes full responsibility for the content
of the manuscript.


\end{document}